\documentclass[11pt]{amsart}
\usepackage{amsmath,amssymb}
\numberwithin{equation}{section}
\usepackage{xcolor}
\usepackage{soul}

\newtheorem{theorem}{Theorem}[section]
\newtheorem{thm}[theorem]{Theorem}

\newtheorem{lem}[theorem]{Lemma}
\newtheorem{prop}[theorem]{Proposition}
\newtheorem{coro}[theorem]{Corollary}
\newtheorem{quest}[theorem]{Question}
\newtheorem{pbm}[theorem]{Problem}

\theoremstyle{definition}
\newtheorem{definition}[theorem]{Definition}
\newtheorem{defi}[theorem]{Definition}
\newtheorem{example}[theorem]{Example}
\newtheorem{exa}[theorem]{Example}

\newtheorem{remark}[theorem]{Remark}
 \newtheorem*{ackn}{Acknowledgements}
 
 \newtheorem*{thmA}{Theorem A} 
 \newtheorem*{thmB}{Theorem B} 
\newtheorem*{thmC}{Theorem C} 
\newtheorem*{thmD}{Theorem D}

 \theoremstyle{plain}
\newtheorem*{namedthm}{\namedthmname}
\newcounter{namedthm}

\makeatletter

 \newcommand{\D}{\mathbb D}
 \newcommand{\R}{\mathbb R}
 
 \newcommand{\C}{\mathbb C}
  \newcommand{\PP}{\mathbb P}

 \newcommand{\e}{\varepsilon}

 \newcommand{\f}{\varphi}
 
 \newcommand{\p}{\psi}
 
 \newcommand{\z}{\overline{z}}

 \newcommand \PSH {{\rm PSH}}

 \usepackage{hyperref}
\hypersetup{
    unicode=false,        
    pdftoolbar=true,      
    pdfmenubar=true,       
    pdffitwindow=false,     
    pdfstartview={FitH},    
    pdftitle={Quasi-psh envelopes 2},    
    pdfauthor={Guedj, Lu},     
    colorlinks=true,       
   linkcolor=blue,          
    citecolor=blue,        
    filecolor=blue,      
    urlcolor=blue}

\subjclass[2010]{32W20, 32U05, 32Q15, 35A23}

\keywords{Monge-Amp\`ere volumes, hermitian metrics}

\begin{document}

\title[Plurisigned hermitian metrics]{Plurisigned hermitian metrics}

\author{Daniele Angella, Vincent Guedj,  Chinh H. Lu}

  \address{Dipartimento di Matematica e Informatica ``Ulisse Dini'' \\ Universit\`a di Firenze \\ viale Morgagni 67/A \\ 50134 Firenze, Italy}

\email{\href{mailto:daniele.angella@gmail.com}{daniele.angella@gmail.com}, \href{mailto:daniele.angella@unifi.it}{daniele.angella@unifi.it}}
\urladdr{\href{https://sites.google.com/site/danieleangella/}{https://sites.google.com/site/danieleangella/}}

\address{Institut de Math\'ematiques de Toulouse   \\ Universit\'e de Toulouse \\
118 route de Narbonne \\
31400 Toulouse, France\\}

\email{\href{mailto:vincent.guedj@math.univ-toulouse.fr}{vincent.guedj@math.univ-toulouse.fr}}
\urladdr{\href{https://www.math.univ-toulouse.fr/~guedj}{https://www.math.univ-toulouse.fr/~guedj/}}

\address{Universit\'e Paris-Saclay, CNRS, Laboratoire de Math\'ematiques d'Orsay, 91405, Orsay, France.}

\email{\href{mailto:hoang-chinh.lu@universite-paris-saclay.fr}{hoang-chinh.lu@universite-paris-saclay.fr}}
\urladdr{\href{https://www.imo.universite-paris-saclay.fr/~lu/}{https://www.imo.universite-paris-saclay.fr/~lu/}}

\date{\today}

 \begin{abstract}
 Let $(X,\omega)$ be a compact hermitian manifold of dimension $n$.
 We study the asymptotic behavior of  Monge-Amp\`ere volumes $\int_X (\omega+dd^c \f)^n$,
 when $\omega+dd^c \f$ varies in the set of hermitian forms that are $dd^c$-cohomologous to $\omega$.
   We show that these Monge-Amp\`ere volumes are uniformly bounded  
 if $\omega$ is "strongly pluripositive",
 and that they are uniformly positive
 if $\omega$ is "strongly plurinegative". 
 This motivates the study of the existence of such plurisigned hermitian metrics.

 We analyze several classes of examples (complex parallelisable manifolds, twistor spaces, Vaisman manifolds)
  admitting such metrics,  showing that they cannot coexist. 
  We take a close look at $6$-dimensional nilmanifolds which admit a left-invariant complex structure,
  showing that each of them admit a plurisigned hermitian metric,
  while only few of them admit a pluriclosed metric.
  We also study $6$-dimensional solvmanifolds with trivial canonical bundle.
\end{abstract}

 \maketitle

\tableofcontents

\section*{Introduction}

 The study of complex Monge-Amp\`ere equations on compact hermitian (non-K\"ahler) manifolds
has gained considerable interest in the last decade.
Tosatti-Weinkove \cite{TW10} and then Sz\'ekelyhidi-Tosatti-Weinkove \cite{STW17}
have resolved the Gauduchon-Calabi-Yau conjecture, extending to the hermitian setting
Yau's fundamental result \cite{Yau78}.
Associated degenerate complex Monge-Amp\`ere equations have been
systematically studied by
 Dinew, Ko{\l}odziej, and Nguyen \cite{DK12,KN15,Din16, KN19},
as well as in  \cite{LPT21,GL21a,GL21b,GL21c}.

\smallskip

By comparison with the setting of K\"ahler manifolds, a key new difficulty lies in the
uniform control of Monge-Amp\`ere volumes.
Given $X$ a compact complex manifold of complex dimension $n$ equipped with a hermitian metric $\omega$, it is of crucial importance to decide whether
$$
v_+(\omega):=\sup \left\{ \int_X (\omega+dd^c \f)^n \; : \; \f \in {\mathcal C}^{\infty}(X)
\text{ and } \omega+dd^c \f>0  \right\}
$$
is finite, and whether
 $$
v_-(\omega):=\inf \left\{ \int_X (\omega+dd^c \f)^n \; : \; \f \in {\mathcal C}^{\infty}(X)
\text{ and } \omega+dd^c \f>0  \right\}
$$
is bounded away from zero.  
Here  $d=\partial+\overline{\partial}$ and $d^c=\frac{1}{2i} (\partial-\overline{\partial})$.

\smallskip

It follows from Stokes theorem  that $v_-(\omega)=v_+(\omega)=\int_X \omega^n$ when $\omega$ is closed
or, more generally, when $dd^c \omega=0$ and $dd^c \omega^2=0$.
The latter conditions are however rather restrictive and it is an important open problem
to decide whether $v_+(\omega)$ (resp. $v_-(\omega)$) is always finite (resp. positive).
We refer the reader to \cite[Theorem C]{GL21b} for an
illustration of how the finiteness of $v_+(\omega)$ is related to a 
transcendental form of Demailly's holomorphic Morse inequalities,
while \cite{GL21c} strongly motivates the condition $v_-(\omega)>0$.

\smallskip

It has been shown in \cite[Theorem A]{GL21b}
that the condition $v_+(\omega)<+\infty$
(resp. $v_-(\omega)>0$) is independent of the choice of hermitian metric
--it only depends on the complex structure-- and is a bimeromorphic invariant.

\smallskip

We further study these conditions in this article,
testing them on various classes of examples.
Our first observation  is the following hereditary result.

\begin{thmA}
{\it 
Let $(X,\omega_X)$ be a compact hermitian manifold and
let $Y \subseteq X$ be a closed submanifold   equipped with a hermitian form $\omega_Y$.
     If $v_+(X,\omega_X)<+\infty$    then $v_+(Y,\omega_Y)<+\infty$.
}
\end{thmA}

We then   establish the finiteness of $v_+(X)$
(resp. positivity of $v_-(X)$)
when $X$ admits special pluripositive (resp. plurinegative) hermitian metrics.

\begin{thmB}
{\it 
Let  $X$ be a compact complex manifold of dimension $n$.

\begin{enumerate}
\item If there exists a hermitian metric $\omega$  and $\e>0$ such that 
  	$dd^c \omega \geq 0$ and
$
dd^c \omega^q \geq \e \omega \wedge dd^c \omega^{q-1},
$
for
$
2 \leq q \leq n-2,
$
then  $v_+(\omega)<+\infty$.

\smallskip
  	
  	\item  If $n=3$ and $X$ admits a  metric  $\omega$ such that  $dd^c \omega\leq 0$,  	 then $v_-(\omega)>0$. 
\end{enumerate}

\noindent 	In particular if $n=3$ and $\omega$ is pluriclosed 
 	 then $0 < v_-(\omega) \leq v_+(\omega)<+\infty$.
}
\end{thmB}

We also provide  a curvature condition to control $v_-(\omega)$ in higher dimension,
see Definition \ref{def:plurineg} and Theorem \ref{thm:plurineghighdim}.

These conditions are  always fulfilled when $\dim_{\C} X \leq 2$, so we initiate   a systematic study
of the $3$-dimensional case.
Using Hahn-Banach theorem in the spirit of \cite{Mic82,HL83},
 one can show 
(Theorems \ref{thm:hbd+} and \ref{thm:hbd-})
that there exists a pluripositive (resp. plurinegative) hermitian metric $\omega$ on $X$  if and only if
any positive current $\tau$ of bidimension $(2,2)$ such that 
$dd^c \tau \leq 0$ (resp. $dd^c \tau \geq 0$) satisfies $dd^c \tau=0$.
Thus in dimension $3$   such plurisigned hermitian metrics cannot coexist,
i.e. the following conditions are mutually exclusive (see Corollary \ref{cor:obstruction}):
\begin{itemize}
 \item $X$ admits a  hermitian metric $\omega$ such that $dd^c \omega \geq 0$ and $dd^c \omega \neq 0$;
 \item $X$ admits a  hermitian metric $\omega$ such that  $dd^c \omega = 0$;
 \item $X$ admits a  hermitian metric $\omega$ such that $dd^c \omega \leq 0$ and $dd^c \omega \neq 0$;
 \item $X$ does not admit any  hermitian metric $\omega$ such that $dd^c \omega$ has a sign.
 \end{itemize}
 
 Each case 
 does occur as we show by analyzing several classes of examples:
 \begin{itemize}
 \item   Complex parallelisable manifolds and twistor spaces of K3 surfaces   admit pluripositive
 hermitian metrics (see Proposition \ref{pro:parallel} and Theorem \ref{thm:twistor}).
 \item The existence of pluriclosed hermitian metrics has been thoroughly studied
 in recent years (see \cite{FPS04,FT09,Verb14,couv,Ot20,ados}).
 \item Vaisman manifolds (a special type of locally conformally K\"ahler manifolds) admit
 plurinegative hermitian metrics (Proposition \ref{pro:vaisman}).
 \item Non-K\"ahler manifolds from the class ${\mathcal C}$ of Fujiki do not admit any plurisgned hermitian metric
 (see Proposition \ref{pro:fujiki}).
 \end{itemize}
 This alternative   is however no longer valid in higher dimension, see Example \ref{exa:different-sign}.

We take a closer look at nilmanifolds $X=\Gamma \backslash G$  of (real) dimension $6$,
where $G$ is a connected and simply connected nilpotent Lie group, and $\Gamma$ is a discrete co-compact subgroup.
There are 34 isomorphism classes of nilpotent Lie algebras in dimension 6, only 18 of which admit a complex structure.
Following \cite{andrada-barberis-dotti,ugarte-transf,ugarte-villacampa,couv} we gather them in four large families (Np), (Ni), (Nii), (Niii) and show the following.

\begin{thmC}
{\it 
 Consider a six-dimensional nilmanifold $X=\Gamma \backslash G$ endowed with a left-invariant complex structure.
There is always a plurisigned hermitian metric. More precisely if $X$ is not a complex torus, then
\begin{itemize}
\item either $X$ belongs to one of the classes (Np), (Nii), (Niii) and 
then any left-invariant hermitian metric is pluripositive but not pluriclosed.
\item or $X$ belongs to the class (Ni) and --depending on the complex structure-- it admits a left-invariant hermitian metric wich is either pluriclosed, or pluripositive but not pluriclosed, or else plurinegative but not pluriclosed.
\end{itemize}
}
\end{thmC}

This analysis largely generalizes the influential work of Fino-Parton-Salamon  \cite{FPS04} who 
 characterized the existence of pluriclosed metric in this context.
We refer the reader to Section \ref{sec:nilmanifolds}
for a more precise statement. 

\smallskip

Following \cite{FOU15} we also analyze the case of $6$-dimensional solvmanifolds, --i.e. compact quotients of a connected solvable Lie group by a discrete subgroup,-- endowed with left-invariant complex structures with holomorphically trivial canonical bundle.
These are gathered in ten large families (see Section \ref{sec:solvmanifolds}),
besides the four large families of nilmanifolds that are dealt with by the previous statement.

\begin{thmD}
{\it 
Consider a 
 six-dimensional solvmanifold 
endowed with a left-invariant complex structure 
with holomorphically-trivial canonical bundle.
Then, there is always a plurisigned metric.
}
\end{thmD}

We refer the reader to Theorem \ref{thm:solvmanifold} for a more precise statement,
which notably shows that in most cases the plurisigned metric is  pluripositive.

\begin{ackn} 
The authors are partially supported by the research projects HERMETIC of the Labex CIMI,
PARAPLUI of the french ANR. The first named-author is supported by
PRIN2017 
(2017JZ2SW5), 
and by GNSAGA of INdAM.
\end{ackn}

\section{Preliminaries}

In the whole article we let $X$ denote a compact complex manifold of complex dimension $n \geq 1$,
and we fix  a hermitian form $\omega_X$ on $X$.

\subsection{Quasi-plurisubharmonic functions}

Let $\omega$ be a semi-positive $(1,1)$-form.

\subsubsection{Quasi-plurisubharmonic functions}

\begin{defi}
 A function is {\em quasi-plurisub\-harmonic} (quasi-psh for short)
  if it is locally given as the sum of  a smooth and a plurisubharmonic function.   
\end{defi}

 Given an open set $U\subset X$, quasi-psh functions
$\f:U \rightarrow \R \cup \{-\infty\}$ satisfying
$
\omega_{\f}:=\omega+dd^c \f \geq 0
$
in the weak sense of currents are called {\em $\omega$-psh} functions on $U$. 
Constant functions are $\omega$-psh functions since $\omega$ is semi-positive.

A ${\mathcal C}^2$-smooth function $u\in \mathcal{C}^2(X)$ has bounded Hessian, hence $\e u$ is
$\omega$-psh on $X$ if $0<\e$ is small enough and $\omega$ is  positive (i.e. hermitian).

\begin{defi}
We let $\PSH(X,\omega)$ denote the set of all $\omega$-plurisubharmonic functions which are not identically $-\infty$.  
\end{defi}

The set $\PSH(X,\omega)$ is a closed subset of $L^1(X)$, 
for the $L^1$-topology.
We refer the reader to \cite{Dem12,GZbook,Din16} for basic
properties of $\omega$-psh functions, and simply
recall that:
\begin{itemize}
\item $PSH(X,\omega) \subset PSH(X,\omega')$ if $\omega \leq \omega'$;
\item $PSH(X,\omega) \subset L^r(X)$ for $r \geq 1$;
 the induced $L^r$-topologies are equivalent;
\item the subset $PSH_A(X,\omega):=\{ u \in PSH(X,\omega), \; -A \leq \sup_X u \leq 0 \}$ is compact in
$L^r(X)$ for any $r \geq 1$ and any $A>0$.
\end{itemize}

\begin{defi}
A quasi-psh function $\f$ has {\em analytic singularities} if it can be locally written as
$$
\f(z)=c \log \sum_{j=1}^s |f_j(z)|^2+\rho(z),
$$
where $c>0$, the $f_j$'s are holomorphic functions and $\rho$ is a smooth function.
\end{defi}

We recall the following fundamental regularization result of Demailly \cite{Dem92}.

\begin{theorem} \label{thm:demreg}
Any quasi-psh function is
the decreasing limit of smooth quasi-psh functions. Moreover when $\omega$ is a hermitian form,
any function $\f \in PSH(X,\omega)$ is the decreasing limit of functions $\f_j \in PSH(X,\omega)$
with analytic singularities.
\end{theorem}

\subsubsection{Monge-Amp\`ere measure}

The complex Monge-Amp\`ere measure 
$
 (\omega+dd^c u)^n 
$
 is well-defined for any
$\omega$-psh function $u$ which is {\it bounded}, as follows from
the theory developed by  Bedford-Taylor in bounded pseudoconvex domains of $\C^n$.

If $\beta=dd^c \rho$ is a K\"ahler form that dominates $\omega$ in a local chart $U$, 
the function $u$ is $\beta$-psh in $U$ hence the positive currents
$(\beta+dd^c u)^j$ are well defined for  $0 \leq j \leq n$
by \cite{BT82}. This allows one to make the following definition.

\begin{defi}
The {\em complex Monge-Amp\`ere measure} of $u$ is
$$
(\omega+dd^c u)^n:=\sum_{j=0}^n \left( \begin{array}{c} n \\ j \end{array} \right) (-1)^{n-j}
(\beta+dd^c u)^j \wedge (\beta-\omega)^{n-j}.
$$
\end{defi}

We refer to \cite{DK12} for an adaptation of the main properties of \cite{BT82} to this
 hermitian and global context.

\smallskip

The mixed Monge-Amp\`ere measures 
$(\omega+dd^c u)^j \wedge (\omega+dd^c v)^{n-j}$ are also
well defined for any $0 \leq j \leq n$, and any bounded $\omega$-psh functions $u,v$.

 A basic property we shall use is the following extension of a fundamental result of Bedford-Taylor:
 
 \begin{lem}
 Let $u,v$ be bounded $\omega$-psh functions, then 
 $\max(u,v) \in PSH(X,\omega)$ and
 $$
 1_{\{u<v\}} (\omega +dd^c \max(u,v))^n= 1_{\{u<v\}} (\omega +dd^c v)^n.
 $$
 \end{lem}
 
 The subtle point here is that the set $\{u<v\}$ is not open in the usual sense
 if $v$ is not continuous,
 it is merely open for the {\it plurifine topology}.

\subsection{Uniform bounds on Monge-Amp\`ere volumes}

Let $(X,\omega)$ be a compact hermitian manifold of complex dimension $n$.
   
      \begin{definition}
 We consider
      $$
    v_{-}(\omega):=\inf \left\{ \int_X (\omega+dd^c u)^n , \; u \in\PSH(X,\omega) \cap {L}^{\infty}(X)  \right\}
    $$
   and
    $$
    v_{+}(\omega):=\sup \left\{ \int_X (\omega+dd^c u)^n  , \; u \in\PSH(X,\omega) \cap {L}^{\infty}(X) \right\}.
    $$
    \end{definition}
    
The supremum and infimum in the definition of $v_{+}(\omega)$ and $v_{-}(\omega)$ can be taken  over 
$\PSH(X,\omega) \cap C^{\infty}(X)$ as follows from 
Theorem \ref{thm:demreg} and
Bedford-Taylor's convergence theorem \cite{BT82}.  
    These quantities have been studied in \cite{GL21b}. 
A major open problem is the following.

\begin{pbm} \label{mainpbm}
Understand whether $v_+(\omega) <+\infty$ and/or $v_-(\omega)>0$.
\end{pbm}

 It follows from Stokes theorem
 that  $0<v_-(\omega)=v_+(\omega)=\int_X \omega^n<+\infty$ when $\omega$ is a K\"ahler form.
 The same result holds true  as soon as the hermitian form satisfies 
 $$
 dd^c \omega=0 
 \; \; \text{ and } \; \;
 d\omega \wedge d^c \omega=0,
 $$
 a vanishing condition  introduced by Guan-Li in \cite{GL10}. 
 
 This condition actually characterizes the preservation of Monge-Amp\`ere masses,
  as was observed by  Chiose in \cite{Chi16}.

\begin{theorem} \cite{Chi16}.\label{thm:chiose}
The following properties are equivalent:
\begin{enumerate}
\item $\int_X (\omega+dd^c \f)^n=\int_X \omega^n$
for all $\f \in PSH(X,\omega) \cap L^{\infty}(X)$.
\item $dd^c \omega=0
\; \; \text{ and } \; \;
d \omega \wedge d^c \omega=0$.
\end{enumerate}
\end{theorem}

When $n=2$ the vanishing of $d \omega \wedge d^c \omega$ is automatic for bidegree reasons, hence
the preservation of Monge-Amp\`ere volumes is equivalent to $\omega$ being a Gauduchon metric.
In higher dimension this condition is quite restrictive and unstable.

Chiose further observed that the Guan-Li condition is moreover equivalent to 
$$
dd^c \omega \geq 0
\; \; \text{ and } \; \;
d \omega \wedge d^c \omega \geq 0.
$$
 Indeed 
 $$
 dd^c(\omega^{n-1})=(n-1)(n-2) d \omega \wedge d^c \omega\wedge\omega^{n-3}+(n-1)dd^c \omega \wedge \omega^{n-2}
 $$
 while Stokes theorem ensures that $\int_X dd^c(\omega^{n-1})=0$.
 
 \begin{remark} \label{rem:balancedvspluriclosed}
 Note that $X$ satisfying the $\partial\overline\partial$-Lemma is not enough to ensure the volume preservation 
 property $v_-(\omega)=v_+(\omega)=\int_X \omega^n$. Some splitting-type complex structures arising as deformations of the holomorphically-parallelizable Nakamura solvmanifold satisfy the $\partial\overline\partial$-Lemma \cite{angella-kasuya-1, aouv}, but they never admit pluriclosed metrics \cite{aouv}, in particular the Guan-Li conditions are never satisfied
 (see also Proposition \ref{pro:fujiki}).
It is a folkore conjecture that balanced metrics always exist on compact complex manifolds satisfying the $\partial\overline\partial$-Lemma \cite{popovici}, and it is also expected that balanced and pluriclosed metrics cannot coexist on a compact complex manifold unless it admits K\"ahler metrics \cite{fino-vezzoni}. 
If these two conjectures were true, the volume preservation property would never be satisfied by compact complex non-K\"ahler manifolds satisfying the $\partial\overline\partial$-Lemma.
 \end{remark}

 In the sequel we 
 analyze several families of compact complex manifolds admitting a hermitian metric $\omega$ such that
  $dd^c \omega \geq 0$ (resp. $dd^c \omega \leq 0$).
 We can  further expect that $d \omega \wedge d^c \omega \leq 0$ (resp. $d \omega \wedge d^c \omega \geq 0$),
 but we should not expect, in general, that these two forms are both globally positive (resp. negative).
 
 \smallskip
  
 We recall the following properties established in \cite{GL21b}.

\begin{theorem} \cite[Proposition 3.2, Theorem 3.7, Theorem 4.12]{GL21b}.
Let $X$ be a compact complex manifold of dimension $n$ and let $\omega$ be a hermitian metric.
\begin{itemize}
\item The condition $v_+(\omega)<+\infty$ is independent of the choice of the hermitian metric $\omega$,
and it is a bimeromorphic invariant.
\item The condition $v_-(\omega) >0$ is independent of the choice of the hermitian metric $\omega$,
and it is a bimeromorphic invariant.
\item If $\alpha \in H^{1,1}_{BC}(X,\R)$ is a nef  class with $\alpha^n>0$,
then $\alpha$ contains a K\"ahler current (hence $X$ belongs to the Fujiki class ${\mathcal C}$) if
and only if $v_+(\omega)<+\infty$.
\end{itemize}
\end{theorem}

Here $H^{1,1}_{BC}(X,\R)$ denotes the first Bott-Chern cohomology group of $X$. The last item is a 
partial answer to an important conjecture of Demailly-Paun \cite{DP04}.
  
  \smallskip
 
 A main goal of this article is to introduce curvature conditions that allow one to 
 partially answer Problem \ref{mainpbm}, and to try and establish them
on various classes of non-K\"ahler manifolds.

      \section{Pluripositive hermitian metrics} \label{sec:stability}

  \subsection{The restriction property} \label{sec:restriction}

We   observe   in this section that     the condition     
$v_+(\omega)<+\infty$ is stable under restriction.

     \begin{thm}  
Let $(X,\omega_X)$ be a compact hermitian manifold and
    let $Y \subset X$ be a closed submanifold of $X$ equipped with a hermitian form $\omega_Y$.
     If $v_+(X,\omega_X)<+\infty$    then $v_+(Y,\omega_Y)<+\infty$.
    \end{thm}
    
    \begin{proof}
    Since the finiteness of $v_+$ is independent of the choice of a hermitian metric, we work here
    with $\omega_Y=(\omega_X)_{|Y}$. 
    
    It follows from Theorem \ref{thm:demreg}
  that one can
    approximate $\f \in PSH(Y,\omega_Y)$ by a decreasing
    sequence of smooth strictly $\omega_Y$-psh functions. Since the complex Monge-Amp\`ere 
    operator is continuous along decreasing sequences \cite{BT82}, it suffices to establish a uniform
    bound from above on $\int_Y (\omega_Y+dd^c \f)^k$, where
    $\f$ is smooth and strictly $\omega_Y$-psh and $k=\dim_{\C}Y$.
    
    It follows from \cite[Proposition 2.1]{CGZ13} that there exists a smooth
   extension $\phi \in PSH(X,\omega_X)$ with $\phi_{|Y}=\f$.
   It is classical (see e.g. \cite[Lemma 2.1]{DP04}) that one can find 
   a function $\p_Y \in PSH(X,\omega_X)$ which is smooth in $X \setminus Y$,
   with analytic singularities along $Y$, and such that 
   $$
   (\omega_X+dd^c \p_Y)^{n-k} \geq \delta_0 [Y],
   $$
   for some   $\delta_0>0$, where $[Y]$ denotes the current
   of integration along $Y$. We infer
   \begin{eqnarray*}
   \int_Y (\omega_Y+dd^c \f)^k 
   &=& \int_X (\omega_X+dd^c \phi)^k \wedge [Y] \\
   & \leq & \delta_0^{-1} \int_X (\omega_X+dd^c \phi)^k \wedge (\omega_X+dd^c \p_Y)^{n-k} \\
   &=& \lim_{j \rightarrow +\infty} \delta_0^{-1} \int_X (\omega_X+dd^c \phi)^k \wedge (\omega_X+dd^c \p_j)^{n-k},
   \end{eqnarray*}
   where $\p_j=\max(\p_Y,-j) \in PSH(X,\omega_X) \cap L^{\infty}(X)$. 
   
   Observe now that   $u_j=\frac{\phi+\p_j}{2} \in PSH(X,\omega_X) \cap L^{\infty}(X)$ with
   $$
   (\omega_X+dd^c \phi)^k \wedge (\omega_X+dd^c \p_j)^{n-k} \leq 2^n (\omega_X+dd^c u_j)^n,
   $$
  and $\int_X (\omega_X+dd^c u_j)^n \leq v_+(X,\omega_X)$, hence
  $$
  v_+(Y,\omega_Y)=\sup_{\f}  \int_Y (\omega_Y+dd^c \f)^k  \leq 2^n \delta_0^{-1} v_+(X,\omega_X)<+\infty,
  $$
  completing the proof.
    \end{proof}

      \subsection{Controlling $v_+$}
   
  We propose here various curvature conditions
  on a hermitian metric $\omega$  that ensure $v_+(\omega)<+\infty$.

     \begin{theorem} \label{thm:higher-dim}
  	Let  $X$ be a compact complex manifold of dimension $n$.
  	Assume there exists a hermitian metric $\omega$ and $\e>0$ such that 
  	\begin{itemize}
  	\item $dd^c \omega \geq 0$ and
  	\item $dd^c \omega^q \geq \e \omega \wedge dd^c \omega^{q-1}$ for $2 \leq q \leq n-2$.
  	\end{itemize}
  	Then  $v_+(\omega)<+\infty$.
  \end{theorem}
  
  In complex dimension $n=3$, the hypothesis boils down to $dd^c \omega \geq 0$.
  We shall provide several examples of manifolds admitting   such  pluripositive
  hermitian metrics in the sequel.
Higher dimensional examples 
satisfying the conditions of Theorem \ref{thm:higher-dim} 
 will be presented in  Example \ref{ex:v+-complex-parallelizable-higher-dim} and Example \ref{exa:higher-dim-plurisigned}.

    \begin{proof}
   When $n=3$ this observation is due to Chiose \cite[Question 0.8]{Chi16}.
  We include a proof 
   as a warm up for the higher dimensional case.
  It follows from Stokes theorem that $\int_X (dd^c u)^3=0$
   for any $u \in \PSH(X,\omega)\cap C^{\infty}(X)$, hence
  \[
  	\int_X (\omega+dd^c u)^3=  \int_X \omega^3
  	+3 \int_X \omega^2 \wedge dd^c u+3 \int_X \omega \wedge (dd^c u)^2.
  	\]
    Integrating by parts we see that
    $$
  \left|   \int_X \omega^2 \wedge dd^c u \right| =\left| \int_X u  {\, dd^c \omega^2} \right|
  \leq B \int_X |u| \omega^3 \leq C
    $$
   by compactness (see \cite[Proposition 8.4]{GZbook}), as we can normalize $u$ by $\sup_X u=0$. On the other hand
   $$
   \int_X \omega \wedge (dd^c u)^2=\int_X -dd^c \omega \wedge du \wedge d^c u \leq 0
   $$
   whenever $dd^c \omega \geq 0$. The result follows.
   
   \smallskip
   
   We now treat the case $n \geq 4$.
   Observe that it suffices to deal with $a\omega$-psh functions, where $0<a \leq 1$ is an arbitrarily small fixed constant.
   Indeed if $u \in PSH(X,\omega) \cap L^{\infty}(X)$, then
   $au \in PSH(X,a\omega) \cap L^{\infty}(X)$ with
   $$
   a^n \omega_u^n \leq ((1-a) \omega+a \omega_u)^n ,
   $$
    hence
   $$
   v_+(\omega) \leq a^{-n} \sup \left\{ \int_X (\omega+a dd^c u)^n, \, u \in PSH(X,\omega) \cap L^{\infty}(X) \right\}.
   $$
   We thus fix  $a>0$ and $u\in PSH(X,\omega)$ with $\sup_X u=0$.
Stokes theorem yields
	\begin{flalign*}
		\int_X (\omega+ a dd^c u)^n &=  \sum_{0\leq p \leq n} \binom{n}{p} a^{p} \int_X \omega^{n-p} \wedge (dd^c u)^p\\
		&= \int_X \omega^n + na\int_X \omega^{n-1}\wedge dd^c u +  \sum_{2\leq p \leq n-1} \binom{n}{p} a^{p} \int_X \omega^{n-p} \wedge (dd^c u)^p\\
		&\leq C_1 - \sum_{2\leq p \leq n-1} \binom{n}{p} a^{p}  \int_X  du \wedge d^cu \wedge (\omega_u-\omega)^{p-2} \wedge dd^c (\omega^{n-p})\\
		&=C_1 + S. 
	\end{flalign*}
	We   decompose the sum $S$ as follows,
	\begin{flalign*}
		S &= \mathop{\sum_{2\leq p \leq n-1}}_{0\leq k\leq p-2}  B_{k,p} a^{p}(-1)^{k+1}  \int_X  du \wedge d^cu \wedge \omega_u^{p-2-k} \wedge \omega^k \wedge dd^c (\omega^{n-p})\\
		&= S_1 +S_2,
	\end{flalign*}
	where $B_{k,p} = \binom{n}{p} \binom{p-2}{k}$,  $S_1$ is the sum for $k$ even and $S_2$	 is the one for $k$ odd.  
	 Using the assumption on $dd^c \omega^{n-p}$ we obtain
	\begin{flalign*}
	S_1&=	-\mathop{\sum_{2\leq p\leq  n-2}}_{0\leq k =2l\leq p-2}  B_{k,p} a^{p}  \int_X  du \wedge d^cu \wedge \omega_u^{p-2-2l} \wedge \omega^{2l} \wedge dd^c (\omega^{n-p})\\
		&\leq-\mathop{\sum_{2\leq p\leq  n-2}}_{0\leq k =2l\leq p-2}  B_{k,p} a^{p}\varepsilon  \int_X  du \wedge d^cu \wedge \omega_u^{p-2-2l} \wedge \omega^{2l+1} \wedge dd^c (\omega^{n-p-1}),
	\end{flalign*}
	while
	\begin{flalign*}
		S_2 &= \mathop{\sum_{3\leq p\leq  n-1}}_{0\leq k =2l+1\leq p-2}  B_{k,p} a^{p}  \int_X  du \wedge d^cu \wedge \omega_u^{p-3-2l} \wedge \omega^{2l+1} \wedge dd^c (\omega^{n-p})\\
		&=\mathop{\sum_{2\leq p\leq  n-2}}_{{ 0\leq k =2l+1\leq p-1 }}  B_{k,p+1} a^{p+1}  \int_X  du \wedge d^cu \wedge \omega_u^{p-2-2l} \wedge \omega^{2l+1} \wedge dd^c (\omega^{n-p-1}),
	\end{flalign*}
	as follows from changing   $p$ in $p-1$. Now $S_1+S_2=$
	$$
	 =\mathop{\sum_{2\leq p\leq  n-2}}_{0\leq k =2l+1\leq q-1}  
	(-B_{k,p} \varepsilon+B_{k,p+1} a )a^{p}
	\int_X  du \wedge d^cu \wedge \omega_u^{p-2-2l} \wedge \omega^{2l+1} \wedge dd^c (\omega^{n-p-1}) \leq 0
	$$
	if  $a$ is small enough, which yields $S_1+ S_2\leq 0$ hence $v_+(\omega) \leq a^{-n}C_1$.
\end{proof}

     \subsection{Pluripositive hermitian metrics} \label{sec:pluripositivecurrents}
     
     \begin{defi}
     A hermitian metric $\omega$ is {\em pluripositive} if  $dd^c\omega\geq0$,
      {\em plurinegative} if $dd^c\omega\leq 0$,  and pluriclosed if $dd^c \omega=0$.
\end{defi}

Pluriclosed metrics are also often called SKT (strongly K\"ahler with torsion) in the literature;
a manifold $X$ is called SKT if it admits a SKT hermitian metric.

\smallskip

   The existence of a pluripositive hermitian metric $\omega$ 
   is a condition that is stable under blow-ups with smooth centers.
   Indeed   if $\pi:Y \rightarrow X$ is the blow up of $X$ with smooth center $Z \subset X$, a hermitian metric
     on $Y$ is obtained by considering 
     $$     \omega_Y=\pi^* \omega_X-\e \theta_Z,     $$
     where $\omega_X$ is  a hermitian metric in $X \setminus Z$ with poles along $Z$,
     $\theta_Z$ is a closed form cohomologous to the current of integration along the exceptional
     divisor $\pi^{-1}(Z)$, and $0<\e$ is small.
     Thus $dd^c \omega_Y=\pi^* dd^c \omega_X$ has the same sign as that of $dd^c \omega_X$.
     
     \smallskip
   
   However this condition is not stable under
   modifications (see Proposition \ref{pro:fujiki}),
and there are   obstructions to the existence of 
     pluripositive hermitian metrics:

    \begin{theorem} \label{thm:hbd+}
    The following properties are equivalent.
    \begin{itemize}
    \item There exists a hermitian form $\omega$ such that $dd^c \omega \geq 0$.
    \item There exists no positive  current $\tau$ of bidimension $(2,2)$ such that $S=dd^c \tau\leq 0$
    with $S \neq 0$.
    \end{itemize}
    \end{theorem}
    
    When the complex dimension is $n=3$, this shows in particular that pluripositive and plurinegative hermitian
    forms can not coexist (see Corollary \ref{cor:obstruction}).
    
      \begin{proof}
 Observe first that the two objects cannot coexist on $X$. Indeed if 
 $\omega$ is a hermitian form such that $dd^c \omega \geq 0$ and 
 $\tau$ is a positive  current of bidimension $(2,2)$ such that $dd^c \tau \leq 0$, we obtain
 $$
 0 \leq \int_X dd^c \omega \wedge \tau=\int_X \omega \wedge dd^c \tau \leq 0,
 $$
 which forces $dd^c \tau=0$ since $\omega$ is a hermitian form.
  Conversely consider 
 $$
 {\mathcal C}:=\left\{\text{negative 
 currents } S   \text{ of bidimension }(1,1)
\text{ with } \int_X S \wedge \omega=-1 \right\}.
 $$
 This is a compact convex set for
 the weak topology of currents.
 We set 
$$
F:=\left\{\text{currents } S=dd^c \tau, 
\text{ where } \tau \geq 0 \text{ has  bidimension } (2,2) \right\}.
$$
 This is a closed set for the weak topology.
  There exists no positive  current $\tau$ of bidimension $(2,2)$ such that $S=dd^c \tau\leq 0$
    with $S \neq 0$ if and only if the sets ${\mathcal C}$ and $F$ are disjoints.
    If such is the case, it follows from Hahn-Banach theorem that we can find a continuous functional
    $\Phi$ on the set of bidimension $(1,1)$ currents that is 
    semi-positive on $F$ and strictly negative on ${\mathcal C}$.
    By deRham duality the functional $\Phi$ is defined by a form $\omega$ of bidegree $(1,1)$. Now
    \begin{itemize}
    \item $\Phi \geq 0$ on $F$ is equivalent to $dd^c \omega \geq 0$, while
    \item $\Phi<0$ on ${\mathcal C}$ is equivalent to $\omega$ being hermitian,
    \end{itemize}
    as follows from a rescaling argument.
   \end{proof}

   \subsection{Twistor spaces} \label{sec:twistor}
   
  Twistor spaces provide a large classe of examples of compact non-K\"ahler manifolds    
which admit a pluripositive hermitian metric.
    
\subsubsection*{Dimension 4}
Let $(M,g)$ be a 
    compact oriented riemannian $4$-manifold.
    The vector bundle $\Lambda^2T^*M$ of $2$-forms can be decomposed as a direct sum
    $
    \Lambda^2T^*M=\Lambda_+ \oplus \Lambda_-,
    $
    where $\Lambda_{\pm}$ denotes the eigenspace of the Hodge $\star$-operator 
    corresponding to the $\pm 1$-eigenvalues of $\star$
    (selfdual and anti-selfdual forms).
    
    The riemannian curvature operator
    $R: \Lambda^2T^*M \rightarrow \Lambda^2T^*M  $ can be decomposed under the action of
    the group of special isometries $SO(4)$ as
    $$
    R=\frac{s}{6}Id+W^-+W^+ +\stackrel{\circ}{r},
    $$
    where $s$ is the scalar curvature, $\stackrel{\circ}{r}$ is the trace free Ricci curvature,
  and  $W^{\pm}$ are the trace free endomorphisms of $\Lambda^{\pm}$.
   The manifold $M$ is called ASD (anti-selfdual) if $W^+=0$;
   this definition is conformally invariant \cite{AHS78}.
   A famous result
  of Taubes \cite{Taub92} provides many examples of such ASD manifolds.
  
  \begin{defi}
  The twistor space $X=X(M,[g])$ of $(M,[g])$ is the total space of the sphere bundle
  of self dual $2$-forms.
  \end{defi}

 There is a natural almost complex structure on $X$ which is integrable if and only if 
 $M$ is ASD \cite{AHS78}. In this case $X$ is a  compact complex   manifold of dimension $3$
 which is never K\"ahler unless $M=S^4$ is the sphere
 (in which case $X=\C\PP^3$), or $M=\C\PP^2$ (in which case 
 $X$ is the flag space of $\C^3$), see \cite{Hit81}.
 
 Despite being non-K\"ahler, twistor spaces have a lot of rational curves, in particular
all fibers $F \sim \PP^1$ of the smooth submersion $\pi:X \rightarrow M$. Note that there also is
a holomorphic projection $\pi:X \rightarrow \PP^1$.

\subsubsection*{Dimension $4n$}

The previous construction generalizes as follows.
      Let $(M,g,D)$ be a quaternionic K\"ahler manifold of  dimension $4n$, i.e.
    an oriented complete $4n$-dimensional riemannian manifold $(M,g)$
    whose holonomy group is contained in the product $Sp(1)Sp(n)$ of quaternionic unitary groups. Such a manifold admits a rank $3$ subbundle $D \subset End(TM)$ invariant by the Levi-Civita connection.
    
      \begin{defi}
  The twistor space $X=X(M,g,D)$ of $(M,g,D)$ is the   bundle of spheres of radius $\sqrt{2}$ of $D$.
  \end{defi}
  
  This is a locally trivial bundle with fibre $S^2$ and structure group $SO(3)$.
  It can be endowed with a natural metric G and an almost complex structure $J$ that is integrable \cite[Theorem 4.1]{Sal82}. Thus $(X,J)$
  is a complex manifold of complex dimension $2n+1$.
    
    \smallskip
    
    When $(M,g)$ is hyperk\"ahler, i.e. when the holonomy   group is contained in the quaternionic unitary group $Sp(n)$, it turns out that
    $(M,g)$ admits three global 
    $g$-orthogonal integrable K\"ahler structures
    $I,J,K$ such that 
    $IJ=-JI=K$.  We thus obtain a pencil of complex structures
    $$
    X(M,g,D) \longrightarrow \PP^1
    $$
    which is integrable. It is called the Calabi family of $(M,g,D)$.

\subsubsection*{Pluripositive hermitian metrics}

 Twistor spaces admit smooth hermitian $(1,1)$-forms $\omega$ that are balanced 
 (i.e. $d \omega^{n-1}=0$, see \cite{Mic82} and \cite[Proposition 4.5]{KV98}). 
 There is a natural hermitian form $\omega$  
 whose curvature has been computed by Kaledin-Verbitsky \cite[Proposition 8.15]{KV98} and
  Deschamps-LeDu-Mourougane \cite[Corollary 5.6]{DLM17}:
    
    \begin{theorem} \label{thm:twistor}
    Let $(M,g,D)$ be a quaternionic K\"ahler manifold of dimension $4n$
    with constant scalar curvature $s \leq 0$. Then 
    the natural hermitian form $\omega$ on the twistor space $ X(M,g,D) $
    is pluripositive $dd^c \omega \geq 0$.
    \end{theorem}

      \section{Plurinegative hermitian metrics}

           \subsection{Monge-Amp\`ere lower bounds} \label{sec:highdim-}

    We  show that the condition    $v_-(\omega)>0$
    is satisfied if $\omega$ satisfies  a  special plurinegative condition.
    
    \begin{defi} \label{def:plurineg}
    We say that a hermitian form $\omega$ satisfies the condition $Plurineg(n)$
    if it is a Gauduchon metric and either $n=2$ or $n \geq 3$ and
    $$
     \sum_{k=1}^{n-\ell-2} (-1)^{n-k-\ell} {n \choose k} {n-k-2 \choose \ell} dd^c\omega^k \wedge \omega^{n-k-2-\ell} \leq 0 
     $$
for any $\ell\in\{0,\ldots,n-3\}$.
    \end{defi}
    
       The condition  $Plurineg(n)$ requires the Gauduchon condition $dd^c \omega^{n-1}=0$.
   It    reduces to the latter when $n=2$, while it moreover asks that
  \begin{itemize}
  \item   $dd^c \omega \leq 0$ when $n=3$;
  \item   $dd^c \omega \leq 0$ and   $3 dd^c \omega^2 -2 \omega \wedge  dd^c \omega  \leq 0$  when $n=4$;
  \item    $dd^c \omega \leq 0$,  
  $dd^c \omega^2 -dd^c \omega \wedge \omega \leq 0$,  and
  $
   2 dd^c \omega^3 - 2 \omega \wedge dd^c \omega^2+\omega^2 \wedge dd^c \omega \leq 0
  $
  when $n=5$.
  \end{itemize}

  \begin{theorem} \label{thm:plurineghighdim}
  	Let  $X$ be a compact complex manifold of dimension $n$.
  	 If $X$ admits a hermitian metric  $\omega$ that satisfies   $Plurineg(n)$
  	 then $v_-(\omega)= \int_X \omega^n>0$. 
 \end{theorem}

  \begin{proof}
  Fix $u$ a smooth $\omega$-psh function.
   We first treat the case $n=3$ to set the scene. The Gauduchon condition yields
$\int_X \omega^2 \wedge dd^c u=0$ while Stokes theorem ensures that $\int_X (dd^c u)^3=0$, thus
\begin{eqnarray*}
\int_X (\omega+dd^c u)^3 & = &\int_X \omega^3+3\int_X \omega^2 \wedge dd^c u+3 \int_X \omega \wedge (dd^c u)^2 
+\int_X (dd^c u)^3 \\
&=&\int_X \omega^3+3\int_X -dd^c \omega \wedge d u \wedge d^c u \\
& \geq & \int_X \omega^3
\end{eqnarray*}
 if $dd^c \omega \leq 0$, so that $v_-(\omega) \geq \int_X \omega^3>0$.
  
To make the arguments  clearer, we also explicitly look at the case $n=4$.
  Using the binomial expansion, the Gauduchon condition and Stokes theorem, we obtain
  \begin{eqnarray*}
  \lefteqn{  \int_X \omega_u^4-\int_X \omega^4 = 6 \int_X -dd^c \omega^2 \wedge du \wedge d^c u
  +4 \int_X -dd^c \omega \wedge dd^c u \wedge du \wedge d^c u} \\
  &=& 2\int_X [-3 dd^c \omega^2+2 dd^c \omega \wedge \omega] \wedge du \wedge d^c u
  +4 \int_X -dd^c \omega \wedge \omega_u \wedge du \wedge d^c u \\
  &\geq & 0
  \end{eqnarray*}
  if $Plurineg(4)$ is satisfied.

We now consider the general case.
\begin{eqnarray*}
\int_X \omega_u^n - \int_X \omega^n
&=& \sum_{k=0}^{n-1} {n \choose k} \int_X \omega^k \wedge (dd^cu)^{n-k} \\
&=& \int_X (dd^cu)^n + \sum_{k=1}^{n-2} {n \choose k} \int_X \omega^k \wedge (dd^cu)^{n-k} + n \int_X u \, dd^c\omega^{n-1} \\
&=& - \sum_{k=1}^{n-2} {n \choose k} \int_X dd^c\omega^k \wedge (dd^cu)^{n-k-2}\wedge du\wedge d^cu ,
\end{eqnarray*}
using the Gauduchon condition $dd^c\omega^{n-1}=0$.
Thus
\begin{eqnarray*}
\lefteqn{
\int_X \omega_u^n - \int_X \omega^n
= - \sum_{k=1}^{n-2} {n \choose k} \int_X dd^c\omega^k \wedge (\omega_u-\omega)^{n-k-2}\wedge du\wedge d^cu }\\
&=& - \sum_{k=1}^{n-2} \sum_{\ell=0}^{n-k-2} (-1)^{n-k-\ell} {n \choose k} {n-k-2 \choose \ell} \int_X dd^c\omega^k \wedge \omega^{n-k-2-\ell}\wedge \omega_u^\ell\wedge du\wedge d^cu \\
&=& - \sum_{\ell=0}^{n-3} \sum_{k=1}^{n-\ell-2} (-1)^{n-k-\ell}{n \choose k} {n-k-2 \choose \ell} \int_X dd^c\omega^k \wedge \omega^{n-k-2-\ell}\wedge \omega_u^\ell\wedge du\wedge d^cu \\
&=& -   \sum_{\ell=0}^{n-3} \int_X   \sum_{k=1}^{n-\ell-2} (-1)^{n-k-\ell} {n \choose k} {n-k-2 \choose \ell} dd^c\omega^k \wedge \omega^{n-k-2-\ell}  \wedge \omega_u^\ell\wedge du\wedge d^cu .
\end{eqnarray*}
Therefore, if
$$ \sum_{k=1}^{n-\ell-2} (-1)^{n-k-\ell} {n \choose k} {n-k-2 \choose \ell} dd^c\omega^k \wedge \omega^{n-k-2-\ell} \leq 0 $$
for any $\ell\in\{0,\ldots,n-3\}$, we get that $v_-(\omega)\geq\int_X \omega^n$.
\end{proof}

       \subsection{Monge-Amp\`ere bounds in dimension 3} \label{sec:dim3}

    We focus 
    in this section on the $3$-dimensional setting and observe that 
    the condition  $v_-(X,\omega_X)>0$
    is satisfied if $X$   admits a plurinegative hermitian metric
    which is not necessarily Gauduchon. Note that the restriction of a plurinegative
    metric yields a plurinegative metric, but the Gauduchon condition is in general not preserved under restriction.

  \begin{theorem} \label{thm:plurinegpos3}
  	Let  $X$ be a compact complex manifold of dimension $3$.
  	 If $X$ admits a hermitian metric  $\omega$ such that  $dd^c \omega\leq 0$, then $v_-(\omega)>0$. 

  	In particular if $\omega$ is pluriclosed, then $0 < v_-(\omega) \leq v_+(\omega)<+\infty$.
  \end{theorem}

  \begin{proof}
Assume by contradiction that there exists a sequence 
 $(u_j)$ such that $u_j \in \PSH(X,\omega)\cap C^{\infty}(X)$, $\sup_X u_j=-1$ and 
  	\[
  	\int_X (\omega+dd^c u_j)^3 \to 0. 
  	\]
  	For each $j$, let $\phi_j\in \PSH(X,\omega)\cap {\mathcal C}^{\infty}(X)$ be the unique solution to 
  	\[
  	(\omega+dd^c \phi_j)^3 = c_j u_j^2 \omega^3
  	\] 
  	normalized by $\sup_X \phi_j =0$.
  	Here $c_j>0$ is a positive constant.
  	 The existence of $\phi_j$ (and that of $c_j$) follows from the main result of \cite{TW10}. 
  	 Observe that  $\int_X |u_j|^4 \omega^3$ is uniformly bounded away from $0$ and infinity. 
  	 Thus by \cite[Lemma 3.13 and Theorem 5.8]{KN15},
  	\[
  	C^{-1}\leq c_j \leq C, \; \phi_j \geq -C, 
  	\]
  	for some uniform constant $C>0$.
  	
  	\smallskip
  	
  	 Let $B$ be a positive constant such that 
  	$
  	d\omega \wedge d^c \omega \leq B \omega^3. 
  	$
  	For $u,v \in PSH(X,\omega) \cap L^{\infty}(X)$
  	we set $\omega_u:=\omega+dd^cu, \omega_v:=\omega+dd^cv$.
  When $dd^c \omega\leq 0$     one obtains
   $$
   dd^c (\omega_u \wedge \omega_v)=dd^c \omega \wedge \omega_u
   +dd^c \omega \wedge \omega_v+2 d\omega \wedge d^c \omega \leq 2B \omega^3,
   $$
  hence
  	\[
  	dd^c (\omega_{u_j}^2+ \omega_{v_j}^2 + \omega_{u_j}\wedge \omega_{v_j}) \leq 6 d\omega \wedge d^c \omega\leq 6B\omega^3. 
  	\]
  	Fix $t>1$ and set $v_j := \max(u_j,\phi_j-t)$. Using Stokes theorem we obtain 
  	\begin{flalign*}
  		\int_X (\omega_{v_j}^3-\omega_{u_j}^3) &= \int_X (v_j-u_j) dd^c (\omega_{u_j}^2+ \omega_{v_j}^2 + \omega_{u_j}\wedge \omega_{v_j})\\
  		&\leq 6B \int_X (v_j-u_j) \omega^3 = 6B\int_{\{u_j<\phi_j-t\}} (v_j-u_j)\omega^3\\ 
  		&\leq 6B \int_{\{u_j<\phi_j-t\}} |u_j|\omega^3 \\
  		&\leq  \frac{6B}{tc_j} \int_{\{u_j<\phi_j-t\}} c_j|u_j|^2\omega^3\\
  		&\leq \frac{6BC}{t} \int_{\{u_j<\phi_j-t\}}  \omega_{\phi_j}^3 \\
  		&\leq \frac{6BC}{t} \int_X  \omega_{v_j}^3.
  	\end{flalign*}
  	In the last line we have used the identity
  	\[{
  {\bf 1}_{\{u_j<\phi_j-t\}} \omega_{\phi_j}^3=	\bf 1}_{\{u_j<\phi_j-t\}} \omega_{v_j}^3 \leq \omega_{v_j}^3.
  	\] 
  	Choosing $t= 12BC$ we obtain $\int_X \omega_{v_j}^3 \leq 2 \int_X \omega_{u_j}^3$. Using \cite[Proposition 3.4]{GL21b}, we arrive at a contradiction since the functions $v_j$ are uniformly bounded. 
  \end{proof}

        \subsection{Various obstructions}

      \subsubsection{Pluriclosed and Plurinegative  metrics}  
      
        Constructing pluriclosed hermitian metrics on compact complex manifolds
   is a problem that has attracted a lot of attention in the last decades.
        We observe here a rigidity property of this condition.
    
    \begin{prop}
    Let $(X,\omega)$ be a compact hermitian manifold such that
    $dd^c \omega=0$. If $dd^c \omega^2 \leq 0$ (resp. $dd^c \omega^2 \geq 0$)
    then $dd^c \omega^j=0$ for all $j \geq 1$.
    \end{prop}
    
    The restrictive condition $dd^c \omega =0$ \& $dd^c \omega^2=0$ has been introduced by Guan-Li
    \cite{GL10}, and further studied by Chiose \cite{Chi16}, it is equivalent to the preservation
    of the Monge-Amp\`eres volumes,
    $v_+(\omega)=v_-(\omega)=\int_X \omega^n$
    (see Theorem \ref{thm:chiose} and Remark \ref{rem:balancedvspluriclosed}).

    \begin{proof}
   It follows from Stokes theorem that  $\int_X dd^c (\omega^{n-1})=0$. Now
   \begin{eqnarray*}
   \frac{dd^c(\omega^{n-1})}{n-1} &=& \left\{ dd^c \omega \wedge \omega+(n-2) d\omega \wedge d^c \omega \right\} \wedge \omega^{n-3} \\
   &=& \left\{ \frac{n-2}{2} dd^c \omega^2  -(n-3) dd^c \omega \wedge \omega \right\} \wedge \omega^{n-3},
   \end{eqnarray*}
   hence
   $$
  \frac{n-2}{2} \int_X dd^c \omega^2 \wedge \omega^{n-3}  =(n-3)\int_X dd^c \omega \wedge \omega^{n-2}.
   $$
   The conclusion follows.
    \end{proof}

    An adaptation of the proof of  Theorem \ref{thm:hbd+} yields the following 
   characterization of the existence of plurinegative metrics \cite[Theorem 3.3]{Eg01}.

        \begin{theorem} \label{thm:hbd-}
   The following properties are equivalent.
    \begin{itemize}
    \item There exists a hermitian form $\omega$ such that $dd^c \omega \leq 0$.
      \item There exists no positive  current $\tau$ of bidimension $(2,2)$ such that $S=dd^c \tau \geq 0$
    with $S \neq 0$.
    \end{itemize}
    \end{theorem}
  
    A similar obstruction for the existence of pluriclosed hermitian metrics   
is well-known \cite[Theorem 3.3]{Eg01}:
    there exists a hermitian form $\omega$ such that $dd^c \omega = 0$
    if and only if the only
   $dd^c$-exact   positive  current $S$ of bidimension $(1,1)$ is $0$.

     \begin{remark}
   It follows from the work of Ivashkovich  that 
   one can extend meromorphic maps with values in
   compact non-K\"ahler manifolds endowed with a plurinegative metric
   (see \cite[Theorem 2.2]{Iv04}).
   \end{remark}

      \subsubsection{Mutually exclusive conditions in dimension $3$}  
    
    Let $X$ be a compact complex $3$-fold. 
     If $X$ admits a pluripositive hermitian metric $\omega$ and another plurinegative hermitian metric $\omega'$,
     then these are actually both pluriclosed. 
     More generally, it follows from Theorems \ref{thm:hbd+} and   \ref{thm:hbd-}
     that  we have the following alternative.
   
   \begin{coro} \label{cor:obstruction}
   Let $X$ be a compact complex $3$-fold.  
 The following conditions are  mutually exclusive:
 \begin{itemize}
 \item $X$ admits a  hermitian metric $\omega$ such that $dd^c \omega \geq 0$ and $dd^c \omega \neq 0$;
 \item $X$ admits a  hermitian metric $\omega$ such that  $dd^c \omega = 0$;
 \item $X$ admits a  hermitian metric $\omega$ such that $dd^c \omega \leq 0$ and $dd^c \omega \neq 0$;
 \item $X$ does not admit any  hermitian metric $\omega$ such that $dd^c \omega$ has a sign.
 \end{itemize}
   \end{coro}
   
   We refer the reader to the examples to follow for an illustration    of each case. 
   
     \begin{proof}
   Assume $\omega$ is a hermitian form such that $dd^c \omega \geq 0$, while 
   $\omega'$ is a hermitian form such that $dd^c \omega' \leq 0$. 
   It follows from Theorem \ref{thm:hbd+} and Theorem \ref{thm:hbd-} that
   $dd^c \omega=0$ and $dd^c \omega'=0$.
   \end{proof}

   Recall that the class ${\mathcal C}$ of Fujiki consists of compact complex
   manifolds that are bimeromorphic to a K\"ahler manifold.
   The non-existence of plurisigned hermitian metric occurs on non-K\"ahler
   Fujiki $3$-folds, as we now observe.
   
      \begin{prop} \label{pro:fujiki}
   Let $X$ be a  compact complex $3$-fold in the Fujiki class ${\mathcal C}$.
   Then $X$ does not admit any plurisigned hermitian metric, unless $X$ is K\"ahler.
   \end{prop}
   
   A celebrated example of a non-K\"ahler $3$-fold bimeromorphic to $\C\PP^3$
   has been provided by Hironaka \cite{Hir60}.
   This shows that \cite[Theorem 6.5]{Eg01} is  {incorrect}.

      \begin{proof}
      This is \cite[Theorem 2.3]{Chi14}. We include the proof for the reader's convenience.
     Let $X$ be a  compact complex $3$-fold in the Fujiki class ${\mathcal C}$.
     Let $\omega$ be a pluripositive hermitian metric.
     There exists a K\"ahler current on $X$, i.e. a positive closed current $T$ of bidegree $(1,1)$ which dominates a hermitian form. Up to rescaling, we can assume $T\geq \omega$.
     
     Then $0 \leq \omega \wedge dd^c \omega \leq T \wedge dd^c \omega$, hence
     $$
     0 \leq \int_X \omega \wedge dd^c \omega \leq \int_X T \wedge dd^c \omega=0,
     $$
     as follows from Stokes theorem. Thus $\omega$ is pluriclosed and it follows from \cite[Theorem 2.2]{Chi14} that $X$ is K\"ahler.
The proof for plurinegative metrics is similar.
\end{proof}

    \subsection{Locally conformally K\"ahler manifolds} \label{sec:hopf}

 We observe in this section that a large class of non-K\"ahler manifolds
 admits a plurinegative hermitian metric $\omega$. This ensures that 
 $v_-(\omega)>0$ when $\dim_{\C} X=3$.
 We then have a closer look at the special subfamily of diagonal Hopf $3$-folds.

  \subsubsection{Existence of plurinegative metrics} 
  
  Recall that a complex manifold $(X,\omega)$ is {\em locally conformally K\"ahler} (lck)
  if one can find local smooth conformal factors $f_j>0$ in an open cover  $\{U_j\}$ 
 of $X$ such that $f_j \omega$ is K\"ahler in $U_j$.
  Thus 
  $$
  d \omega=-\frac{d f_j}{f_j} \wedge \omega=d(-\log f_j) \wedge \omega
  $$
  and these glue together into a globally well defined closed $1$-form
  $\theta$ such that $d\omega =\theta \wedge \omega$. 
  
  The {\em Lee form} $\theta$ is unique and defines a conformal invariant,
 in particular  it vanishes if and only if (a conformal multiple of) the metric $\omega$  is K\"ahler.
 We refer the reader to \cite{DO98} for an introduction to lcK geometry,
 and to  \cite{OV11,OV20, bazzoni, ornea-verbitsky-book} for a more recent account.
 
 \begin{defi}
 One says that  a compact hermitian manifold $(X,\omega)$    is {\it locally conformally K\"ahler with potential}
 if  there exists a smooth positive   plurisubharmonic function $\f:\tilde{X} \rightarrow \R_+^*$
 on the universal cover $\tilde{X}$ of $X$
 such that $\pi^* \omega=\frac{dd^c \f}{\f}$
 and $\frac{\f \circ f}{\f}=c_f$ is constant for all deck transformations $f$.
 
 The manifold $(X,\omega)$ is {\it Vaisman} if $\nabla \theta=0$, where $\nabla$ denotes the 
 Levi-Civita connection associated to $\omega$.  Vaisman manifolds are lck with potential.
 \end{defi}
 
The $1$-form $\theta$ decomposes as $\theta=\alpha+\overline{\alpha}$ where
$\alpha$ is a $(1,0$)-form which is $\partial$-closed and such that $\overline{\partial} \alpha=-\partial \overline{\alpha}$.
Thus
$$
d \omega \wedge d^c \omega=i \alpha \wedge \overline{\alpha} \wedge \omega^2 \geq 0.
$$
When $n=\dim_{\C} X \geq 3$, we observe moreover that this $(3,3)$-form is non zero
if $\alpha \neq 0$, which is  the case if $X$ is not K\"ahler.

\smallskip

Since $d \omega \wedge d^c \omega \geq 0$   does not vanish, we have observed  after
Theorem \ref{thm:chiose} that $dd^c \omega$ cannot be positive.
 We therefore   investigate whether    $\omega$ is plurinegative.

  \begin{prop} \label{pro:vaisman}
 Let $(X,\omega)$ be a compact hermitian manifold   which is {\it lcK with potential}.
Let $\f:\tilde{X} \rightarrow \R_+^*$ be a smooth positive  plurisubharmonic  function 
 on the universal cover of $X$
 such that $\pi^* \omega=\frac{dd^c \f}{\f}$.
 
 If $(X,\omega)$ is Vaisman then $\p=\log \f$ is plurisubharmonic and for all $k \geq 1$
 $$
 dd^c \omega^k =-k(dd^c \p)^{k+1} \leq 0.
 $$
 In particular $dd^c \omega \leq 0$ and 
\begin{itemize}
\item $v_-(X,\omega)>0$ if $\dim_{\C} X=3$;
\item  there is no pluripositive hermitian metric $\tilde{\omega}$ on $X$.
\end{itemize}

Conversely if $\p=\log \f$ is plurisubharmonic  then $(X,\omega)$ is Vaisman.
 \end{prop}

 The last statement of this Proposition is due to Ornea-Verbitsky \cite[Corollary 2.4]{OV20}.
 
 \begin{proof}
 Slightly abusing notation we identify $\omega$ with the invariant form $\pi^* \omega$.
 We set $\p=\log \f$ and observe that
 $$
 dd^c \p=\frac{dd^c \f}{\f}-\frac{d\f \wedge d^c \f}{\f^2}
 \; \; \text{ and } \; \;
 d \p \wedge d^c \p=\frac{d\f \wedge d^c \f}{\f^2}
 $$
 hence $ \omega=dd^c \p+d\p \wedge d^c \p$.
 Observe that $d\p \wedge d^c \p$ has rank $1$, while $\omega$ has rank $n$,
 so the rank of $dd^c \p$ is at least $n-1$.
 
 Since $d\p \wedge d^c \p$ has rank $1$, we obtain
 $$
 \omega^k=(dd^c \p)^k+k(dd^c \p)^{k-1} \wedge d\p \wedge d^c \p
 $$
 and
 \begin{equation} \label{eq:lck}
  dd^c \omega^k=k(dd^c \p)^{k-1} \wedge dd^c (d\p \wedge d^c \p)=-k(dd^c \p)^{k+1} \leq 0
 \end{equation}
 if $\p=\log \f$ is psh.
  For $k=1$ we obtain $dd^c \omega \leq 0$, so it
     follows from Theorem \ref{thm:plurinegpos3} that $v_-(\omega)>0$   when 
   $\dim_{\C} X=3$.
 
 Assume now that $\tilde{\omega}$ is a hermitian metric on $X$ such that
 $dd^c \tilde{\omega} \geq 0$. It follows from Stokes theorem that
 $$
 0 \leq \int_X dd^c \tilde{\omega} \wedge \omega^{n-2} =\int_X  \tilde{\omega} \wedge dd^c \omega^{n-2} \leq 0,
 $$
 hence $\tilde{\omega}$ is pluriclosed and $-dd^c \omega^{n-2}=(dd^c \p)^{n-1} \equiv 0$.
 The latter equality implies that $dd^c \p$ has rank $\leq n-2$, a contradiction.
 
 \smallskip
 
 It remains to understand when $\p=\log \f$ is plurisubharmonic.
 Recall that $d \omega=\theta \wedge \omega$ with $\theta=-d\p$.
 Set $\theta^c:=\frac{1}{2i}(\alpha-\overline{\alpha})=-d^c \p$ so that
 $d^c \omega=\theta^c \wedge \omega$. Thus
 $$
 \beta:=dd^c \p=-d \theta^c=\omega-\theta \wedge \theta^c
 $$
 is a real $(1,1)$-form whose  eigenvalues with respect to $\omega$ are $1$, with multiplicity $(n-1)$,
 and $\lambda=1-|\theta|_{\omega}^2$, since
 $
 \beta^n=\omega^n -n \omega^{n-1} \wedge \theta \wedge \theta^c
 =\left[1-|\theta|_{\omega}^2 \right] \omega^n.
 $
 Thus $\p$ is psh if and only if $|\theta|_{\omega}^2 \leq 1$.
 
 When $X$ is Vaisman then $|\theta|_{\omega}^2$ is constant and 
 there exists a unique conformal choice of $\omega$
 such that $|\theta|_{\omega}^2 \equiv 1$ (see \cite[Remark p232]{V82}).
 In this case $\p$ is psh hence $\beta \geq 0$ and $\beta^n \equiv 0$.
 Conversely it follows from \eqref{eq:lck} that
 $$
 \beta^n=(dd^c \p)^n=-\frac{1}{n-1} dd^c \omega^{n-1}.
 $$
 Stokes theorem ensures that $\int_X \beta^n =0$.
 Thus either $|\theta|_{\omega}^2 \equiv 1$,
 then $X$ is Vaisman and $\omega$ is a Gauduchon metric
 \cite[Corollary 2.4]{OV20},
 or $\lambda$ changes signs and $\p$ is not plurisubharmonic.
 \end{proof}

\begin{remark}
The existence of plurinegative metrics on Vaisman manifolds can also be deduced from the Ornea-Verbitsky Embedding theorem for Vaisman manifolds \cite{OV10} and the Example \ref{exa:hopf} below.
More generally, manifolds admitting locally conformally K\"ahler metrics with potential (this class includes Vaisman manifolds too) can be holomorphically embedded into linear Hopf manifolds \cite{OV10}. Therefore the problem concerning the existence of plurinegative hermitian metrics on locally conformally K\"ahler manifolds with potential is reduced to study linear Hopf manifolds (which are Vaisman if and only if the generator of the fundamental group is diagonalizable, see e.g. \cite[Theorem 16.3]{ornea-verbitsky-book}).
\end{remark}

 \begin{exa}\label{exa:hopf}
Let $X$ be a diagonal Hopf manifold $X=\C^n \setminus \{0\} / \sim$, where we identify $z$ and $A z$ in $\C^n \setminus \{0\}$,
   for some diagonal matrix $A={\rm Diag}(\lambda_1,\ldots,\lambda_n)$ with entries $\lambda_j \in \D^*$.
    We choose $\alpha_i:=\frac{\log 2}{2 (-\log |\lambda_i|)}$  and set
   $$
   \f(z)=\sum_{i=1}^n |z_i|^{2\alpha_i},
\; \;    \text{ so that } \; \; 
   \p=\log \f \in PSH(\C^n).
   $$
    Then $\f \circ A(z)=\frac{1}{2} \f(z)$ hence
   $ 
   \omega(z):=\frac{dd^c \f(z)}{\f(z)}=dd^c \p(z)+d\p \wedge d^c \p(z)
   $
   defines a hermitian form on $X$ such that   $   dd^c \omega \leq 0$.
   \end{exa}

      \subsubsection{The classical Hopf $3$-fold}

      We consider here $X=\C^3 \setminus \{0\} / \sim$, where we identify $z$ and $\lambda z$ in $\C^3 \setminus \{0\}$,
   for some $\lambda \in \D^*=\{ \zeta \in \C, 0<|\zeta|<1\}$.
      There is  a natural holomorphic map $f:X \rightarrow \PP^{2}$ with elliptic fibers.
  
  Set $\eta=\sum_{j=1}^n \z_j dz_j=\partial |z|^2$.
   The invariant $(1,1)$-form   $|z|^{-4}i \eta \wedge \overline{\eta}$
      compensates the lack of positivity of $f^* \omega_{FS}$
   so that the invariant $(1,1)$ form
   $$
   \omega=f^* \omega_{FS}+|z|^{-2}i \eta \wedge \overline{\eta}=
   \frac{\sum_{j=1}^n i dz_j \wedge d \z_j}{|z|^2}=\frac{dd^c \f}{\f}
   $$
    induces a hermitian form
   on $X$ that we still denote by $\omega$. 
   Since 
    $\log \f=\log|z| \in PSH(\C^3\setminus \{0\})$,
   it follows from Proposition \ref{pro:vaisman}   that $dd^c \omega \leq 0$.

 \smallskip

 We would like to test the finiteness of $v_+(\omega)$.
     It follows from the proof of Theorem \ref{thm:plurinegpos3}.1  that there exists $C>0$ such that for any
   $u \in PSH(X,\omega)$ normalized by $\sup_X u=0$, we have
   $$
   3 \int_X (-dd^c \omega) \wedge du \wedge d^c u-C \leq \int_X (\omega+dd^c u)^3 
   \leq  3 \int_X (-dd^c \omega) \wedge du \wedge d^c u+C.
   $$
   
   It is thus tempting to think that one can reach $v_+(\omega)=+\infty$ by constructing
   $\omega$-psh functions whose gradient does not belong to $L^2$.
   There are several   functions $v \in PSH(\PP^2,\omega_{FS})$ whose gradient does not belong
   to $L^2$;
   they induce
   functions 
   $$
   u=v \circ f \in PSH(X,f^*\omega_{FS}) \subset PSH(X,\omega)
   $$
   with gradient $\nabla u \notin L^2(X)$. However 
   $d v \circ f $ is proportional to $\eta$,    hence
   $$
   \int_X (-dd^c \omega) \wedge d (v \circ f)  \wedge d^c (v \circ f)=0.
   $$

     One can construct singular $\omega$-psh functions that do not come from $\PP^2$
     as follows: 
     the function $\rho=\log dist_{\omega}(\cdot,p)$
      is psh   near $p \in X$ and smooth off $p$.
     We multiply it by a cut-off function $\chi$
     which is identically equal to $1$ near $p$. For $\e>0$
     small, the function $\f=\e \chi \rho$ belong to $PSH(X,\omega)$
     and it has a logarithmic singularity at $p$. However  $\nabla \f \in L^2(X)$
     as the singularity of $\f$ is isolated.
     One could consider a convergent series of such functions, this would produce
     examples with a discrete set of logarithmic singularities (and possibly an
     uncountable polar set $(\f=-\infty)$), but their global gradient
     would still belong to $L^2$.

  \begin{quest}
  Does one have $v_+(\omega)<+\infty$?
  \end{quest}

   We tend to expect a positive answer to this problem, but a negative one would be quite
   interesting as well!

    \section{Homogeneous examples}  \label{sec:pluripositive}

We finally study  several classes of compact (locally) homogeneous manifolds.

     \subsection{Complex parallelizable manifolds}
     
     Let $X$ be a compact complex manifolds of dimension $n$ that
     is {\it complex parallelizable}, i.e. such that the holomorphic
     tangent bundle $T^{1,0}X$ is holomorphically trivial.
       It has been shown by Wang \cite{Wan54} that $X=\Gamma\backslash G$ is the quotient
     of a connected and simply connected complex Lie group $G$ by 
     a discrete subgroup $\Gamma$, and that
     complex tori (for which $G=\C^n$) are the only ones that are K\"ahler. On the other side, they always admits balanced metrics by \cite[Proposition 3.1]{abbena-grassi} and they cannot admit pluriclosed metrics, see \cite[page 7110]{fino-grantcharov-vezzoni}.

     \begin{prop} \label{pro:parallel}
    A compact complex parallelizable manifold $X$ admits a 
    hermitian metric $\omega$ such that $dd^c \omega^k \geq 0$ for all $k \geq 1$.
    
      In particular, $v_+(X,\omega)<+\infty$ when when $\dim_{\C} X = 3$. Moreover $X$ does not admit any plurinegative hermitian metric unless it is a complex torus.
     \end{prop}
     
     Complex parallelizable manifolds that arise as quotients of complex semisimple Lie groups
      (e.g. compact quotients of $SL(2;\mathbb C)$ by a lattice) have been studied in \cite{yachou}.
          For a classification of  compact complex parallelizable solvmanifolds in complex dimension $3$
     --as well as a partial classification in dimension $4$ and $5$--,
     we refer the reader to \cite[Theorem 1 and Section 6]{nakamura}.

   \begin{proof}
     By definition there exist 
     $\f_1,\ldots,\f_n \in H^0(X,\Omega_X^1)$ holomorphic
     $1$-forms that are linearly independent at each point. We set
     $$
     \omega:=\sum_{j=1}^n i \f_j \wedge \overline{\f_j}.
     $$
     This is  a hermitian $(1,1)$-form such that 
     $$
     dd^c \omega=i \partial \overline{\partial }\omega=\sum_{j=1}^n i^2 \partial \overline{\partial }(\f_j \wedge \overline{\f_j}).
     $$
     Since $\overline{\partial} \f_j=0$ and $\partial \overline{\f_j}=0$ we obtain, setting $\eta_j=\partial \f_j$,
  $$
  dd^c \omega=\sum_{j=1}^n -i^2 \eta_j \wedge \overline{\eta_j} \geq 0.
  $$
  Recall indeed that if $\eta$ is  a $(2,0)$-form then $-i^2 \eta \wedge \overline{\eta}=i^4 \eta \wedge \overline{\eta}$
  is a weakly positive $(2,2)$-form (that is strongly positive if 
  $\eta=\alpha_1 \wedge \alpha_2$ is decomposable).
  
  \smallskip
  
For $k \geq 1$, we decompose 
  $$
  \omega^k=\sum_{j_1,\ldots,j_k=1}^n i \f_{j_1} \wedge \overline{\f_{j_1}} \wedge \cdots \wedge  i \f_{j_k} \wedge \overline{\f_{j_k}}
  =\sum_{J=(j_1,\ldots,j_k)  } i^{k^2}  \alpha_J \wedge \overline{\alpha_J},
  $$
  where $\alpha_J=\f_{j_1} \wedge \cdots \wedge \f_{j_k} $ is a holomorphic $k$-form. Thus
  $$
  i \partial \overline{\partial}\omega^k=
  \sum_{J=(j_1,\ldots,j_k)  } i^{1+k^2} (-1)^k \beta_J \wedge \overline{\beta_J},
  $$
  with $\beta_J=\partial \alpha_J$.
  Since $i^{1+k^2} (-1)^k=i^{(k+1)^2}$, we conclude that $dd^c \omega^k \geq 0$.
  
  Observe that $dd^c \omega \neq 0$ unless 
  $\omega$ is K\"ahler, in which case  $X$ is a complex torus.
  It follows therefore from Theorem \ref{thm:plurinegpos3} and Theorem \ref{thm:hbd-} that
  $v_+(X,\omega)<+\infty$ when $\dim_{\C} X = 3$, and
   $X$ does not admit a plurinegative hermitian metric unless it is a torus.
  \end{proof}

In dimension higher than $3$, the conditions needed to control $v_+$ as in Theorem \ref{thm:higher-dim} are not trivially satisfied. For example, take the complex parallelizable manifold of complex dimension $n=4$ characterized by a coframe of holomorphic $1$-forms $(\varphi^j)_{j\in\{1,2,3,4\}})$ with structure equations
$$ 
d\varphi^1=d\varphi^2=d\varphi^3=0, \quad d\varphi^4=-\varphi^2\wedge\varphi^3 ,
$$
see \cite[Type IV.2, page 108]{nakamura}. 
The metric $\omega=\sum_{j=1}^{4}\sqrt{-1}\varphi^j\wedge\bar\varphi^j$ satisfies:
\begin{eqnarray*}
dd^c\omega &=& \sqrt{-1} \varphi^2\wedge\bar\varphi^2 \wedge \sqrt{-1} \varphi^3\wedge\bar\varphi^3 \geq0, \\
dd^c\omega^2 &=& 2\,\sqrt{-1} \varphi^1\wedge\bar\varphi^1\wedge \sqrt{-1}\varphi^2\wedge\bar\varphi^2\wedge\sqrt{-1}\varphi^3 \wedge\bar\varphi^3 \geq0, 
\end{eqnarray*}
but
\begin{eqnarray*}
dd^c\omega^2-\varepsilon dd^c\omega\wedge\omega &=&
(2-\varepsilon)\,\sqrt{-1} \varphi^1\wedge\bar\varphi^1\wedge \sqrt{-1}\varphi^2\wedge\bar\varphi^2\wedge\sqrt{-1}\varphi^3 \wedge\bar\varphi^3 \\
&& -\varepsilon\,\sqrt{-1} \varphi^2\wedge\bar\varphi^2\wedge \sqrt{-1}\varphi^3\wedge\bar\varphi^3\wedge\sqrt{-1}\varphi^4 \wedge\bar\varphi^4 
\end{eqnarray*}
 is never semi-positive for $ \e>0$.

\smallskip

The following is an example of a complex parallelizable manifold admitting metrics with the property of Theorem \ref{thm:higher-dim}: if a lattice can be provided for the Lie group, then we would get a compact manifold with $v_+<\infty$.

\begin{example}\label{ex:v+-complex-parallelizable-higher-dim}
Consider the complex parallelizable solvmanifold of complex dimension $n=4$ of type IV.5 in \cite[page 108]{nakamura}, which is characterized by a coframe of holomorphic $1$-forms with structure equations
$$ d\varphi_1=0, \quad d\varphi^2=\varphi^1\wedge\varphi^2, $$
$$ d\varphi^3=\alpha\varphi^1\wedge\varphi^3, \quad d\varphi^4=-(1+\alpha)\varphi^1\wedge\varphi^4, $$
depending on a parameter $\alpha\in\mathbb R\setminus\{-1,0\}$.
It is not known whether it admits lattices, see \cite[page 110]{nakamura}.
Consider the left-invariant metric
$$
\omega = a_{1} \, \sqrt{-1}\varphi^1\wedge\bar\varphi^1 + a_{2}   \, \sqrt{-1}\varphi^2\wedge\bar\varphi^2 + a_{3} \, \sqrt{-1}\varphi^3\wedge\bar\varphi^3 + a_{4} \, \sqrt{-1}\varphi^4\wedge\bar\varphi^4 ,
$$
where $a_1,a_2,a_3,a_4>0$.
A straightforward computation yields
\begin{eqnarray*}
dd^c\omega
&=& a_{2} \, \sqrt{-1}\varphi^1\wedge\bar\varphi^1 \wedge \sqrt{-1}\varphi^2\wedge\bar\varphi^2 + a_{3} |\alpha|^2 \, \sqrt{-1}\varphi^1\wedge\bar\varphi^1 \wedge \sqrt{-1}\varphi^3\wedge\bar\varphi^3 \\
&& + a_{4} |\alpha + 1|^2 \, \sqrt{-1}\varphi^1\wedge\bar\varphi^1 \wedge \sqrt{-1}\varphi^4\wedge\bar\varphi^4 \geq 0 ,\\
dd^c\omega^2
&=& 2 \, a_{2} a_{3} |\alpha + 1|^2 \, \sqrt{-1}\varphi^1\wedge\bar\varphi^1 \wedge  \sqrt{-1}\varphi^2\wedge\bar\varphi^2 \wedge  \sqrt{-1}\varphi^3\wedge\bar\varphi^3 \\
&& + 2 \, a_{2} a_{4} |\alpha|^2 \, \sqrt{-1}\varphi^1\wedge\bar\varphi^1 \wedge \sqrt{-1}\varphi^2\wedge\bar\varphi^2 \wedge \sqrt{-1}\varphi^4\wedge\bar\varphi^4 \\
&& + 2 \, a_{3} a_{4} \sqrt{-1}\varphi^1\wedge\bar\varphi^1 \wedge \sqrt{-1}\varphi^3\wedge\bar\varphi^3 \wedge \sqrt{-1}\varphi^4\wedge\bar\varphi^4 \geq 0.
\end{eqnarray*}
Therefore
\begin{eqnarray*}
\lefteqn{ dd^c\omega^2-\varepsilon\omega\wedge dd^c\omega } \\
&=&
\left( - |\alpha|^2 \varepsilon + 2 |\alpha + 1|^2 - \varepsilon \right) a_{2} a_{3} \sqrt{-1}\varphi^1\wedge\bar\varphi^1 \wedge \sqrt{-1}\varphi^2\wedge\bar\varphi^2 \wedge \sqrt{-1}\varphi^3\wedge\bar\varphi^3 \\
&& + \left( - |\alpha + 1|^2 \varepsilon + 2 |\alpha|^2 - \varepsilon \right) a_{2} a_{4} \sqrt{-1}\varphi^1\wedge\bar\varphi^1 \wedge \sqrt{-1}\varphi^2\wedge\bar\varphi^2 \wedge \sqrt{-1}\varphi^4\wedge\bar\varphi^4 \\
&& + \left( - |\alpha + 1|^2 \varepsilon - |\alpha|^2 \varepsilon + 2 \right) a_{3} a_{4} \sqrt{-1}\varphi^1\wedge\bar\varphi^1 \wedge \sqrt{-1}\varphi^3\wedge\bar\varphi^3 \wedge \sqrt{-1}\varphi^4\wedge\bar\varphi^4 .
\end{eqnarray*}
Therefore, if we take
$$ 
0<\varepsilon \leq \min\left\{ \frac{2|\alpha+1|^2}{1+|\alpha|^2}, \frac{2|\alpha|^2}{1+|\alpha+1|^2}, 
\frac{2}{|\alpha|^2+|1+\alpha|^2} \right\} , 
$$
we obtain $dd^c\omega^2-\varepsilon\omega\wedge dd^c\omega\geq0$.
\end{example}

    \subsection{Six-dimensional nilmanifolds} \label{sec:nilmanifolds}

    In this section, we consider {\em nilmanifolds}, namely, compact quotients $\Gamma\backslash G$ of connected simply-connected nilpotent Lie groups $G$ by co-compact discrete subgroups $\Gamma$.
We recall that a Lie group $G$ is called {\em nilpotent} if its associated Lie algebra $\mathfrak g$ satisfies that the lower central series $\{\mathfrak{g}_j:=[\mathfrak{g},\mathfrak{g}_{j-1}]\}_{j\in\mathbb N}$, with $\mathfrak{g}_0:=\mathfrak{g}$, eventually vanishes.
In dimension $6$, according to \cite{morozov, magnin}, there are only $34$ isomorphism classes of nilpotent Lie algebras over $\mathbb R$, and $10$ of them are reducible.

We consider {\em left-invariant complex structures} on $\Gamma \backslash G$, i.e. complex structures that are induced by complex structures on $G$ being invariant under left-translations. Equivalently, left-invariant complex structures correspond to linear complex structures on the corresponding Lie algebra satisfying an integrability condition. According to \cite{salamon}, only $18$ of the above $34$ nilpotent Lie algebras admit left-invariant complex structures.
We notice that the existence of lattices in nilpotent Lie groups is well-understood thanks to \cite{malcev}, more precisely it corresponds to having rational constant structures, a condition that is satisfied in all the above considered cases.

Among the latter, only four may admit a hermitian metric $\omega$ which is pluriclosed, as shown by Fino-Parton-Salamon in \cite{FPS04}.
We show in this section that the remaining 14 classes all admit a pluripositive hermitian metric, 
and we also analyze whether there exists a balanced metric, i.e. a hermitian
metric $\omega$ such that $d \omega^{2}=0$.

\smallskip

Left-invariant complex structures on six-dimensional nilmanifolds are gathered in four families
in \cite{andrada-barberis-dotti, ugarte-transf, ugarte-villacampa, couv}, 
some of them depending on continuous parameters, up to {\em linear} equivalence.
These families are described by a coframe of left-invariant $(1,0)$-forms $(\varphi^1,\varphi^2,\varphi^3)$ with structure equations as follows (we use the short-hands $\varphi^{1\bar2}:=\varphi^1\wedge\bar\varphi^2$, etc.).
\begin{description}
\item[(Np)] $d\varphi^1=d\varphi^2 = 0$, $d\varphi^3=\rho\, \varphi^{12}$ where $\rho\in\{0,1\}$; 
these are the {\em complex parallelizable} structures \cite{Wan54, nakamura}. 
The manifold is a compact quotient of a {\em complex} Lie group by a discrete subgroup; 
$\rho=0$ corresponds to the complex torus, 
while $\rho=1$ corresponds to the Iwasawa manifold \cite{fernandez-gray}.

\smallskip

\item[(Ni)] $d\varphi^1=d\varphi^2 = 0$, $d\varphi^3= \rho\, \varphi^{12} + \varphi^{1\bar1} + \lambda\,\varphi^{1\bar 2} + D\,\varphi^{2\bar2}$ where $\rho\in\{0,1\}$, $\lambda\in\mathbb R^{\geq0}$, $D\in\mathbb C$ with $\Im D\geq0$;
this class (and the following) contains {\em nilpotent} complex structures: the ascending series $\{\mathfrak{a}_j(J):=\{X \in \mathfrak{g} : [X,\mathfrak{g}]+[JX,\mathfrak{g}]\subseteq \mathfrak{a}_{j-1}\}\}_{j\in\mathbb N}$, with $\mathfrak{a}_0(J):=0$, eventually equals $\mathfrak{g}$. The case $\rho=0$ corresponds to {\em Abelian} complex structures, namely, the subalgebra of left-invariant $(1,0)$-vector fields is Abelian.

\smallskip

\item[(Nii)] $d\varphi^1=0$, $d\varphi^2=\varphi^{1\bar1}$, $d\varphi^3=\rho\varphi^{12}+B\,\varphi^{1\bar2}+c\,\varphi^{2\bar1}$ where $\rho\in\{0,1\}$, $B \in \mathbb C$, $c\in\mathbb R^{\geq0}$ with  $(\rho,B,c)\neq(0,0,0)$.

\smallskip

\item[(Niii)] $d\varphi^1=0$, $d\varphi^2 =\varphi^{13} + \varphi^{1\bar 3}$, $d\varphi^3=\sqrt{-1}\rho\, \varphi^{1\bar1}\pm \sqrt{-1}(\varphi^{1\bar2} - \varphi^{2\bar1})$ where $\rho\in\{0,1\}$; this class contains non-nilpotent complex structures.
\end{description}

We refer to \cite{andrada-barberis-dotti, ugarte-transf, ugarte-villacampa, couv} and \cite[Table 1]{aouv2}
 for more details on the underlying Lie algebras.

We can restrict to study {\em left-invariant} hermitian structures on $\Gamma \backslash G$, namely, hermitian structures on $G$ being invariant by the whole action of $G$ by left-translations. 
Such structures correspond to {\em linear} hermitian structures on the associated Lie algebra.
Indeed 
Belgun Symmetrization Trick \cite[Theorem 7]{belgun} (see also \cite[Theorem 2.1]{fino-grantcharov}) provides the following.

\begin{lem}[{\cite[Proposition 3.6]{ugarte-transf}}]\label{lem:average}
Let $\Gamma \backslash G$ be a compact quotient of a Lie group $G$ endowed with a left-invariant complex structure. If it admits a plurisigned hermitian metric, then it admits a {\em left-invariant} plurisigned hermitian metric.
\end{lem}

\begin{proof}
The case of pluriclosed metrics is \cite[Proposition 3.6]{ugarte-transf}. 
The same argument applies here; we include a proof for the readers' convenience.

Since $G$ admits a compact quotient, it is {\em unimodular} \cite[Lemma 6.2]{milnor}, namely its left-invariant Haar measure $\mu$ is also right-invariant. We consider the symmetrization map
$$ \mu \colon \wedge^{\bullet} (\Gamma\backslash G) \to \wedge^\bullet \mathfrak{g}^* , $$
$$
 \mu(\alpha)(X_1,\ldots,X_k):=\int_{\Gamma\backslash G} \alpha(m)(X_1(m), \ldots, X_k(m)) \mu(m) , 
 $$
where $\wedge^\bullet \mathfrak{g}^*$ is identified with the subspace of   left-invariant forms. It is clear that $\mu|_{\wedge^\bullet \mathfrak{g}^*}=\mathrm{id}$ and that $d\circ\mu=\mu\circ d$.
Since $J$ is left-invariant, we have $\mu\circ J=J\circ \mu$. Therefore, there hold also $\partial\circ\mu=\mu\circ\partial$ and $\overline\partial\circ\mu=\mu\circ\overline\partial$.

In particular, if $\omega$ is a hermitian structure, then $\mu(\omega)$ is a left-invariant $(1,1)$-form. It is straightforward to check that $\mu(\omega)>0$ is still a hermitian structure.
We also have: $\partial\overline\partial(\mu(\omega)) = \mu(\partial\overline\partial\omega)$. In particular
if $\omega$ is pluripositive (resp. plurinegative), then $\mu(\omega)$ is pluripositive (resp. plurinegative). Indeed, assume that $\omega$ is pluripositive. Then, for any $(1,0)$-form $\eta$
$$
dd^c\omega  \wedge\sqrt{-1}\eta\wedge\bar\eta=c\sqrt{-1}\varphi^{1\bar1}\wedge\sqrt{-1}\varphi^{2\bar2}\wedge\sqrt{-1}\varphi^{3\bar3}
$$
 for $c\geq0$. We can restrict to left-invariant $(1,0)$-forms $\eta$, since positivity is a pointwise notion. Then, thanks to \cite[Lemma 2.5]{angella-kasuya-1}, $\mu(\omega)\wedge\eta=\mu(\omega\wedge\eta)=\mu(c)\sqrt{-1}\varphi^{1\bar1}\wedge\sqrt{-1}\varphi^{2\bar2}\wedge\sqrt{-1}\varphi^{3\bar3}$ where $\mu(c)\geq0$.
\end{proof}

We therefore only consider left-invariant hermitian structures. With respect to a chosen coframe,
 it is straighforward to check that they are of the form
\begin{eqnarray}\label{eq:metric}
2\,\omega &=& \sqrt{-1} r^2\, \varphi^{1\bar1} + \sqrt{-1} s^2\,\varphi^{2\bar2} + \sqrt{-1} t^2\, \varphi^{3\bar3} \\
&& + u\, \varphi^{1\bar2} - \bar u\, \varphi^{2\bar1} + v\, \varphi^{2\bar3} - \bar v\, \varphi^{3\bar2} + z\, \varphi^{1\bar3} - \bar z\, \varphi^{3\bar1} \nonumber
\end{eqnarray}
where $r,s,t\in \mathbb R$, $u,v,z\in\mathbb C$ satisfy
$$ r^2>0 ,\qquad s^2>0 , \qquad t^2 > 0 , $$
$$
r^2s^2 > |u|^2 , \qquad
r^2t^2 > |z|^2 , \qquad
s^2t^2 > |v|^2 ,
$$
$$
r^2s^2t^2 + 2 \Re(\sqrt{-1}\,\bar u z \bar v) > t^2|u|^2+r^2|v|^2+s^2|z|^2 ,
$$
in order for the metric to be positive-definite.

We compute $\sqrt{-1}\partial\overline\partial\omega$ for the above families.
Note that positivity and strong positivity are equivalent for $(2,2)$-forms in dimension $3$
(see \cite[p. 279--280]{Mic82}),
and that it is enough to test positivity pairing with left-invariant forms, since it is a pointwise notion.
Recall that, by \cite{benson-gordon,hasegawa}, K\"ahler metrics do not exist on non-tori nilmanifolds (even with complex structures that are possibly non-left-invariant).

\begin{description}
\item[(Np)] For 
complex paralellizable type, $\omega$ is always balanced and
$$ \sqrt{-1}\partial\overline\partial\omega = \frac{1}{2} \rho^2t^2\sqrt{-1}\varphi^{1\bar1}\wedge\sqrt{-1}\varphi^{2\bar2} $$
is always strongly-positive. More precisely, any left-invariant hermitian metric on the torus is pluriclosed (in fact, K\"ahler), and any left-invariant hermitian metric on the Iwasawa manifold is pluripositive, balanced, non-pluriclosed.
The underlying Lie algebras are respectively $\mathfrak h_1$ and $\mathfrak h_5$, in the notation of \cite{salamon}, to which we refer for more details.

\smallskip

\item[(Ni)] For 
nilpotent type in Family I, the balanced condition is equivalent to
$$ 
s^{2} t^{2} + {\left(r^{2} t^{2} - |z|^2 \right)} D + {\left(-\sqrt{-1} \, t^{2} \overline{u} + v \overline{z}\right)} \lambda - |v|^2 = 0 
$$
(it is satisfied on $\mathfrak{h}_2$, $\mathfrak{h}_3$, $\mathfrak{h}_4$, $\mathfrak{h}_5$, $\mathfrak{h}_6$ for metrics such that $r^2 = 1$, $v = z = 0$, and $s^2 + D = \sqrt{-1}\, \bar u\, \lambda$)
and
$$
\sqrt{-1}\partial\overline\partial\omega =
\frac{t^2}{2} \left( \lambda^{2} + \rho^{2} - 2\Re D \right)  \sqrt{-1}\varphi^{1\bar1}\wedge\sqrt{-1}\varphi^{2\bar2} .
$$
Any hermitian metric is either 
\begin{itemize}
\item pluripositive non-pluriclosed (when $2\Re D<\lambda^2+\rho^2$, which happens on Lie algebras $\mathfrak h_2$, $\mathfrak h_3$, $\mathfrak h_4$, $\mathfrak h_5$, $\mathfrak h_6$),
\item  or pluriclosed (when $2\Re D=\lambda^2+\rho^2$, which happens on Lie algebras $\mathfrak h_2$, $\mathfrak h_4$, $\mathfrak h_5$, $\mathfrak h_8$), 
\item or plurinegative non-pluriclosed (when $2\Re D>\lambda^2+\rho^2$, which happens on Lie algebras $\mathfrak h_2$, $\mathfrak h_3$, $\mathfrak h_4$, $\mathfrak h_5$).
\end{itemize}

\item[(Nii)] for the complex structures of nilpotent type in Family II, the balanced condition is never satisfied, and
$$
\sqrt{-1}\partial\overline\partial\omega =
\frac{t^2}{2} \, \big( c^{2} + \rho^{2} + |B|^2 \big) \sqrt{-1}\varphi^{1\bar1} \wedge \sqrt{-1}\varphi^{2\bar2}
$$
is always strongly-positive, and never zero.
The underlying Lie algebras are $\mathfrak h_7$, $\mathfrak h_9$, $\mathfrak h_{10}$, $\mathfrak h_{11}$, $\mathfrak h_{12}$, $\mathfrak h_{13}$, $\mathfrak h_{14}$, $\mathfrak h_{15}$, $\mathfrak h_{16}$.

\item[(Niii)] For the complex structures of nilpotent type in Family II, the balanced condition is equivalent to
$$
\left\{\begin{array}{l}
t^{2} \Re u + \Im (z \overline{v}) = 0 \\
\sqrt{-1} \, s^{2} \overline{z} + \overline{u} \overline{v} = 0 \\
\rho = 0
\end{array}
\right.
$$
(it is satisfied on $\mathfrak{h}_{19}^{-}$ for metrics such that $u=z=0$)
and
$$
\sqrt{-1}\partial\overline\partial\omega =
 t^{2} \sqrt{-1} \varphi^{1\bar1} \wedge \sqrt{-1}\varphi^{2\bar2} + s^{2} \sqrt{-1} \varphi^{1\bar1} \wedge \sqrt{-1} \varphi^{3\bar3}
$$
is always strongly-positive, and never zero.
The underlying Lie algebras are $\mathfrak h_{19}^-$, $\mathfrak h_{26}^+$.
\end{description}

The following statement summarizes the previous discussion, it can be seen as an extension of 
\cite{FPS04} to plurisigned metrics.

\begin{theorem} \label{thm:nilmanifolds}
Consider a six-dimensional nilmanifold $X$ endowed with a left-invariant complex structure
and assume $X$ is not a complex torus.
Then, there is always a plurisigned metric. More precisely:
\begin{itemize}
\item either $X$ belongs to one of the families  (Np), (Nii) and (Niii) then any left-invariant hermitian metric is pluripositive but not pluriclosed.
\item or $X$ belongs to (Ni) and --depending on the complex structure--, every left-invariant hermitian metric is either 
pluriclosed, or pluripositive but not pluriclosed, or else plurinegative but not pluriclosed.
\end{itemize}
\end{theorem}

By Corollary \ref{cor:obstruction}  the latter three conditions are mutually exclusive.

\subsection{Higher dimensional nilmanifolds}

As noticed in \cite{FPS04}, the pluriclosed condition for left-invariant metrics on six-dimensional nilmanifolds 
depends only on the complex structure, and we noticed the same behavior for the plurisigned condition in 
Theorem \ref{thm:nilmanifolds}.

This is no longer true in higher dimension for the pluriclosed condition \cite[Remark 4.1]{enrietti-fino-vezzoni}. More precisely, $8$-dimensional nilmanifolds with left-invariant complex structures admitting pluriclosed metrics are classified into two families in \cite[Section 4]{enrietti-fino-vezzoni}. For the second family, the pluriclosed condition only depends on the complex structure, while for the first family it also involves the parameters of the metric.
We make a similar observation  for the plurisigned condition.

\begin{exa}\label{exa:different-sign}
We consider here the eight-dimensional nilmanifold with left-invariant complex structure characterized by the structure equations
$$ d\varphi^1=0, \quad d\varphi^2=0, \quad  d\varphi^3=\varphi^1\wedge\bar\varphi^1+\frac{1}{2}\varphi^2\wedge\bar\varphi^2, \quad  d\varphi^4=-\varphi^1\wedge\bar\varphi^2,$$
with respect to a left-invariant coframe $\{\varphi^1,\varphi^2,\varphi^3,\varphi^4\}$ of $(1,0)$-forms. It belongs to the first family in the above mentioned classification, more precisely, it corresponds to parameters $B_4=1$, $C_4=\frac12$, $F_5=-1$, the others zero, in the notation of \cite{enrietti-fino-vezzoni}. We consider a Hermitian metric of the diagonal form
$$ \omega = \sqrt{-1} \, a_{1}  \varphi^{1} \wedge \bar\varphi^{1} + \sqrt{-1} \, a_{2} \varphi^{2} \wedge \bar\varphi^{2} + \sqrt{-1} \, a_{3} \varphi^{3} \wedge \bar\varphi^{3} + \sqrt{-1} \, a_{4} \varphi^{4} \wedge \bar\varphi^{4} , $$
where $a_1,a_2,a_3,a_4>0$. Observe that
$$ 
dd^c\omega=( -a_{3} + a_{4} )  \sqrt{-1}\varphi^{1} \wedge \bar\varphi^1\wedge\sqrt{-1}\varphi^{2} \wedge \bar\varphi^{2},
$$
showing that they can be either pluripositive, or pluriclosed, or plurinegative, 
depending on the value of $a_3/a_4$.

Note that this example does not satisfy the condition of Theorem \ref{thm:higher-dim}. Indeed,
\begin{eqnarray*}
\lefteqn{dd^c\omega^2-\varepsilon dd^c\omega\wedge\omega} \\
&=& \left(-\varepsilon a_3(a_4-a_3)+2a_3a_4\right) \sqrt{-1}\varphi^1\wedge\bar\varphi^1\wedge\sqrt{-1}\varphi^2\wedge\bar\varphi^2\wedge\sqrt{-1}\varphi^3\wedge\bar\varphi^3 \\
&& +\left(-\varepsilon a_4(a_4-a_3)-2a_3a_4\right)\sqrt{-1}\varphi^1\wedge\bar\varphi^1\wedge\sqrt{-1}\varphi^2\wedge\bar\varphi^2\wedge\sqrt{-1}\varphi^4\wedge\bar\varphi^4
\end{eqnarray*}
is not positive.
 \end{exa}

 We now provide an $8$-dimensional example which
satisfies the curvature conditions of Theorem \ref{thm:higher-dim}.

\begin{exa}\label{exa:higher-dim-plurisigned}
We consider again an example among the eight-dimensional nilmanifolds in the first family of \cite{enrietti-fino-vezzoni}. 
More precisely, take parameters $B_1= 1$, $G_3= 1$, the others zero, that is
consider 
the structure equations
$$ d\varphi^1=0, \quad d\varphi^2=0, \quad d\varphi^3=\varphi^1\wedge\varphi^2, \quad d\varphi^4=\varphi^2\wedge\bar\varphi^1.$$
Take the diagonal metric
$$ \omega = \sqrt{-1} \, a_{1}  \varphi^{1} \wedge \bar\varphi^{1} + \sqrt{-1} \, a_{2} \varphi^{2} \wedge \bar\varphi^{2} + \sqrt{-1} \, a_{3} \varphi^{3} \wedge \bar\varphi^{3} + \sqrt{-1} \, a_{4} \varphi^{4} \wedge \bar\varphi^{4} , $$
where $a_1,a_2,a_3,a_4>0$.
We compute
\begin{eqnarray*}
dd^c\omega
&=& ( a_{3} + a_{4} ) \sqrt{-1}\varphi^{1} \wedge \bar\varphi^{1} \wedge \sqrt{-1}\varphi^{2} \wedge \bar\varphi^{2} \geq 0 , \\
dd^c\omega^2 &=& 2 \, a_{3} a_{4} \sqrt{-1} \varphi^{1} \wedge \bar\varphi^{1} \wedge \sqrt{-1} \varphi^{2} \wedge \bar\varphi^{2} \wedge \sqrt{-1} \varphi^{3} \wedge \bar\varphi^{3} \\
&& + 2 \, a_{3} a_{4} \sqrt{-1} \varphi^{1} \wedge \bar\varphi^{1} \wedge \sqrt{-1} \varphi^{2} \wedge \bar\varphi^{2}
 \wedge \sqrt{-1} \varphi^{4} \wedge \bar\varphi^{4} ,
\end{eqnarray*}
hence
\begin{eqnarray*}
\lefteqn{ dd^c\omega^2-\varepsilon dd^c\omega\wedge\omega } \\
&=&
a_{3} ( -\varepsilon a_{3}   +(2-\varepsilon) a_{4}  ) \sqrt{-1} \varphi^{1} \wedge \bar\varphi^{1} \wedge \sqrt{-1} \varphi^{2} \wedge \bar\varphi^{2} \wedge \sqrt{-1} \varphi^{3} \wedge \bar\varphi^{3} \\
&& + a_{4} ( (2-\varepsilon) a_{3} - \varepsilon a_{4} ) \sqrt{-1} \varphi^{1} \wedge \bar\varphi^{1} \wedge \sqrt{-1} \varphi^{2} \wedge \bar\varphi^{2} \wedge \sqrt{-1} \varphi^{4} \wedge \bar\varphi^{4} ,
\end{eqnarray*}
which is non-negative for suitable choices of $a_3,a_4,\varepsilon$
 (take e.g. $a_3=a_4=1$, $\e\leq\frac{1}{2}$).
In particular, $v_+(\omega)<+\infty$.
\end{exa}

 \subsection{Six-dimensional solvmanifolds with trivial canonical bundle} \label{sec:solvmanifolds}
 
 We now consider  {\it solvmanifolds} --i.e. compact quotients of a connected solvable Lie group by a closed subgroup--
 of real dimension $6$
  admitting a left-invariant complex structure with holomorphically trivial canonical bundle.
  
According to \cite{FOU15} such complex structures 
 are either nilmanifolds as above (see \cite[Theorem 1.3]{salamon}   and \cite[Theorem 2.7]{barberis-dotti-verbitsky})
  or belong to one of the classes below. We fix a coframe $\{\varphi^1,\varphi^2,\varphi^3\}$ of left-invariant $(1,0)$-forms
  and use the same notations as in the previous section.
\begin{description}
\item[(Si)] $d\varphi^1=A\varphi^{13}+A\varphi^{1\bar3}$, $d\varphi^2=-A\varphi^{23}-A\varphi^{2\bar3}$, $d\varphi^3=0$, where $A = \cos\theta+{\sqrt{-1}}\sin\theta, \theta\in[0,\pi)$. The complex structures in this family are of splitting type in the sense of \cite[Assumption 1]{Kas13}, see \cite{aouv}.

\smallskip

\item[(Sii)] $d\varphi^1=0$, $d\varphi^2=-\frac{1}{2}\varphi^{13}-\big(\frac{1}{2}+\sqrt{-1} x\big) \varphi^{1\bar3}+\sqrt{-1}x\,\varphi^{3\bar1}$, and
$d\varphi^3=\frac{1}{2}\varphi^{12}+\big(\frac{1}{2}-\frac{\sqrt{-1}}{4x}\big)\varphi^{1\bar2}+
\frac{\sqrt{-1}}{4x}\varphi^{2\bar1}$, where $x\in\mathbb R^{>0}$.

\smallskip

\item[(Siii1)] $d\varphi^1=\sqrt{-1}(\varphi^{13}+\varphi^{1\bar3})$,
 $d\varphi^2=-\sqrt{-1}(\varphi^{23}+\varphi^{2\bar3})$, $d\varphi^3=\pm \varphi^{1\bar1}$.

\smallskip

\item[(Siii2)] $d\varphi^1=\varphi^{13}+\varphi^{1\bar3}$, $d\varphi^2=-\varphi^{23}-\varphi^{2\bar3}$, $d\varphi^3=\varphi^{1\bar2}+\varphi^{2\bar1}$.

\smallskip

\item[(Siii3)] $d\varphi^1=\sqrt{-1}(\varphi^{13}+\varphi^{1\bar3})$, 
$d\varphi^2=-\sqrt{-1}(\varphi^{23}+\varphi^{2\bar3})$, $d\varphi^3=\varphi^{1\bar1}+\varphi^{2\bar2}$.

\smallskip

\item[(Siii4)] $d\varphi^1=\sqrt{-1}(\varphi^{13}+\varphi^{1\bar3})$, 
$d\varphi^2=-\sqrt{-1}(\varphi^{23}+\varphi^{2\bar3})$, $d\varphi^3=\pm(\varphi^{1\bar1}-\varphi^{2\bar2})$.

\smallskip

\item[(Siv1)] $d\varphi^1=-\varphi^{13}, \quad d\varphi^2=\varphi^{23},\quad d\varphi^3=0$. This case corresponds to the holomorphically-parallelizable {\em Nakamura manifold} \cite{nakamura}.	This and the next two following cases are complex structures of splitting type in the sense of \cite[Assumption 1]{Kas13}, see \cite{aouv}.

\smallskip

\item[(Siv2)] $d\varphi^1=2\sqrt{-1}\varphi^{13}+\varphi^{3\bar3}$, $d\varphi^2=-2\sqrt{-1}\varphi^{23}+x\,\varphi^{3\bar3}$, $d\varphi^3=0$, where $x \in \{0,1\}$.

\smallskip

\item[(Siv3)] $d\varphi^1=A\,\varphi^{13}-\varphi^{1\bar3}$, $d\varphi^2=-A\,\varphi^{23}+\varphi^{2\bar3}$, $d\varphi^3=0$, where $A\in\mathbb C$ with $|A|\neq1$. 

\smallskip

\item[(Sv)] $d\varphi^1=-\varphi^{3\bar3}$, $d\varphi^2=\frac{\sqrt{-1}}{2}\varphi^{12}+\frac{1}{2}\varphi^{1\bar3}-\frac{\sqrt{-1}}{2}\varphi^{2\bar1}$, 
$d\varphi^3=\frac{\sqrt{-1}}{2}(-\varphi^{13}+\varphi^{3\bar1})$.
\end{description}

We refer the reader to  \cite{otal-thesis, FOU15} and \cite[Table 2]{aouv2} for more details.

\medskip

Thanks to Lemma \ref{lem:average}, we can focus our attention to left-invariant hermitian structures, which are of the form \eqref{eq:metric}.
We compute $\sqrt{-1}\partial\overline\partial\omega$ case by case.

\begin{description}
\item[Si] $\sqrt{-1}\partial\overline\partial\omega=2(-1)\big( (\Re A)^{2} r^{2} \varphi^{1\bar1} + \sqrt{-1}(\Im A)^{2} u \varphi^{1\bar2} - \sqrt{-1}(\Im A)^{2} \overline{u} \varphi^{2\bar1} + (\Re A)^{2} s^{2} \varphi^{2\bar2} \big) \wedge \varphi^{3\bar3}$.
 When $(\cos\theta)^4\geq(\sin\theta)^4$, any metric is pluripositive; otherwise, there are pluripositive metrics (taking $|u|^2\leq \frac{\cos\theta}{\sin\theta}r^2s^2$). When $\cos\theta=0$, there exist   pluriclosed metrics (for $u=0$) and K\"ahler metrics (for $u=v=z=0$, see \cite[Theorem 5.1.3]{otal-thesis}). 
Plurinegative metrics never exists, while balanced metrics always exist (when $v=z=0$).

\smallskip

\item[Sii] $\sqrt{-1}\partial\overline\partial\omega=\left( \frac{4 \, t^{2} x^{2} + t^{2}}{16 \, x^{2}} \right) 
(-1)\varphi^{1\bar1}\wedge \varphi^{2\bar2} + 
\left( s^{2} x^{2} + \frac{1}{4} \, s^{2} \right) 
(-1)\varphi^{1\bar1}\wedge \varphi^{3\bar3}$.
 Any metric is pluripositive.
There is neither pluriclosed metrics, nor plurinegative metrics,  nor K\"ahler metrics. 
Balanced metrics always exist (take $u=v=z=0$).

\smallskip

\item[Siii1] $\sqrt{-1}\partial\overline\partial\omega=2\big( \sqrt{-1} \, u  \sqrt{-1}\varphi^{1\bar2} - \sqrt{-1} \, \overline{u} \sqrt{-1}\varphi^{2\bar1} \big) \wedge \sqrt{-1}\varphi^{3\bar3}$.
There exist pluriclosed metrics (take $u=0$), but neither K\"ahler metrics nor balanced metrics.

\smallskip

\item[Siii2] $\sqrt{-1}\partial\overline\partial\omega=t^{2} (-1)\varphi^{1\bar1} \wedge \varphi^{2\bar2} 
+ 2 \, r^{2} (-1) \varphi^{1\bar1} \wedge \varphi^{3\bar3} 
+ 2 \, s^{2} (-1)\varphi^{2\bar2} \wedge \varphi^{3\bar3}$.
Every metric is pluripositive. There is neither pluriclosed, nor plurinegative,
nor K\"ahler metrics. Balanced metrics 
exist (take $u\in\mathbb R$ and $v=z=0$).

\smallskip

\item[Siii3] $\sqrt{-1}\partial\overline\partial\omega=-t^{2} (-1)\varphi^{1\bar1} \wedge \varphi^{2\bar2} +
 2 (-1)\sqrt{-1} \big(  u  \varphi^{1\bar2} -  \overline{u}  \varphi^{2\bar1} \big) \wedge  \varphi^{3\bar3}$.
 There are neither pluripositive, nor pluriclosed metrics. 
There exists plurinegative metrics (take $u=0$), but neither K\"ahler nor balanced metrics.

\smallskip

\item[Siii4] $\sqrt{-1}\partial\overline\partial\omega=t^{2}  (-1)\varphi^{1\bar1} \wedge \varphi^{2\bar2} 
+ 2 (-1)\big( \sqrt{-1} \, u  \varphi^{1\bar2} - \sqrt{-1} \, \overline{u}  \varphi^{2\bar1} \big) \wedge \varphi^{3\bar3}$.
 Pluriclosed metrics and plurinegative metrics never exist. 
There are pluripositive metrics (take $u=0$), but no K\"ahler metrics.
There are balanced metrics (characterized by parameters $r^2=s^2$, $v=z=0$).

\smallskip

\item[Siv1] $\sqrt{-1}\partial\overline\partial\omega=\frac{(-1)}{2} \big( r^{2} \varphi^{1\bar1} + \sqrt{-1} \, u  \varphi^{1\bar2} - \sqrt{-1} \, \overline{u}  \varphi^{2\bar1} + s^{2}  \varphi^{2\bar2} \big) \wedge \varphi^{3\bar3}$. 
Every metric is pluripositive, there are neither pluriclosed nor plurinegative metrics. There is no K\"ahler metric, and every metric is balanced.

\smallskip

\item[Siv2] $\sqrt{-1}\partial\overline\partial\omega=2 (-1)\big( r^{2}  \varphi^{1\bar1} + \sqrt{-1} \, u  \varphi^{1\bar2} - \sqrt{-1} \, \overline{u} \varphi^{2\bar1} + s^{2} \varphi^{2\bar2} \big) \wedge \varphi^{3\bar3}$. 
Every metric is pluripositive, there are neither pluriclosed nor plurinegative metrics. There are no K\"ahler metrics and no balanced metrics.

\smallskip

\item[Siv3] $\sqrt{-1}\partial\overline\partial\omega=\frac{(-1)}{2} \big( |A - 1|^2 r^{2} \varphi^{1\bar1} + \sqrt{-1} \, |A + 1|^2 u \varphi^{1\bar2} -\sqrt{-1} \, |A + 1|^2 \overline{u} \varphi^{2\bar1} + |A - 1|^2 s^{2} \varphi^{2\bar2}\big)\wedge \varphi^{3\bar3}$.
There are pluripositive metrics: when $\Re A\leq0$, any metric is pluripositive; when $\Re A>0$, pluripositive metrics are characterized by $|u|^2 \leq \left|\frac{A-1}{A+1}\right|^4r^2s^2$. 
There are neither pluriclosed nor plurinegative, nor  K\"ahler metrics. Balanced metrics always exist 
(take $v=z=0$). 

\smallskip

\item[Sv] $\sqrt{-1}\partial\overline\partial\omega=\frac{(-1)}{2} \varphi^{1\bar1} \wedge \big( \sqrt{-1} \, v  \varphi^{2\bar3} - \sqrt{-1} \, \overline{v} \varphi^{3\bar2} + \frac{1}{4} \, s^{2} \varphi^{3\bar3}\big)$. There are pluripositive metrics (take $v=0$).
There are neither pluriclosed, nor plurinegative, nor K\"ahler, nor balanced metrics. 
\end{description}

Notice that it is no longer true that the plurisigned property of left-invariant metrics is completely determined by the complex structure.

\smallskip

We summarize the previous dicussion in the following:

\begin{theorem} \label{thm:solvmanifold}
Let $X$ be a non-K\"ahler 
six-dimensional solvmanifold endowed with a left-invariant complex structure 
with holomorphically-trivial canonical bundle.
Then $X$ admits a plurisigned hermitian metric.
More precisely:
\begin{itemize}
\item either $X$ belongs to the class (Siii1) and there are pluriclosed   metrics;
\item or $X$ belongs to  (Siii3) and there are plurinegative non pluriclosed metrics;
\item or else $X$ belongs to any of the remaining classes, and it admits a left-invariant pluripositive hermitian metric that is not pluriclosed.
\end{itemize}
\end{theorem}

Recall that these three conditions are mutually exclusive by Corollary \ref{cor:obstruction}.
The only K\"ahler example in the above list correspond to a special Lie algebra
${\mathcal G}={\mathcal A}_{5,17}^{0,0,1} \oplus \R$ of splitting type
(see \cite[Theorem 2.18]{FOU15}).

   \begin{remark}
If a hermitian metric is both balanced and pluriclosed, then it is K\"ahler \cite[Remark 1]{alexandrov-ivanov}.
It is conjectured \cite[Problem 3]{fino-vezzoni} that a compact complex manifold admitting both pluriclosed metrics and balanced metrics also admits K\"ahler metrics. 
The conjecture is confirmed for the above nilmanifolds and solvmanifolds 
by \cite[Theorem 6.3 and Theorem 6.4]{fino-vezzoni}.
We notice that these facts fail when replacing pluriclosed with plurisigned: 
the Iwasawa manifold does not admit any K\"ahler metric \cite{fernandez-gray, benson-gordon, hasegawa}, but every left-invariant hermitian metric on it is both balanced and pluripositive.
\end{remark}


\begin{thebibliography}{99}

\bibitem[AG86]{abbena-grassi}
E. Abbena, A. Grassi, \emph{Hermitian left invariant metrics on complex Lie groups and cosymplectic Hermitian manifolds}, Boll. Un. Mat. Ital. A (6) \textbf{5} (1986), no. 3, 371--379.

  \bibitem[AI01]{alexandrov-ivanov} B. Alexandrov, S. Ivanov, 
  \emph{Vanishing theorems on Hermitian manifolds,}
   Differential Geom. Appl. \textbf{14} (2001), no. 3, 251--265.

\bibitem[ABD11]{andrada-barberis-dotti} A. Andrada, M. L. Barberis, I. G. Dotti, 
\emph{Classification of abelian complex structures on $6$-dimensional Lie algebras, }
 J. Lond. Math. Soc. (2) \textbf{83} (2011), no. 1, 232--255. 
 Corrigendum: 
  J. Lond. Math. Soc. (2) \textbf{87} (2013), no. 1, 319--320.

\bibitem[ADOS22]{ados} D.~Angella, A.~Dubickas, A.~Otiman, J.~Stelzig,
\emph{On metric and cohomological properties of Oeljeklaus-Toma manifolds}, preprint arXiv:2201.06377.


\bibitem[AK17]{angella-kasuya-1} D. Angella, H. Kasuya, 
\emph{Bott-Chern cohomology of solvmanifolds, }
 Ann. Global Anal. Geom. \textbf{52} (2017), no. 4, 363--411. 

\bibitem[AOUV17]{aouv} D. Angella, A. Otal, L. Ugarte, R. Villacampa, 
\emph{Complex structures of splitting type, } 
Rev. Mat. Iberoam. \textbf{33} (2017), no. 4, 1309--1350.

\bibitem[AOUV]{aouv2} D. Angella, A. Otal, L. Ugarte, R. Villacampa, 
\emph{On Gauduchon connections with K\"ahler-like curvature,}
 to appear in  Commun. Anal. Geom.
 
 \bibitem[AHS78]{AHS78} M.F.Atiyah, N.J.Hitchin, I.M.Singer,
\emph{Self-duality in four-dimensional Riemannian geometry.}, 
Proc. Roy. Soc. London Ser. A 362 (1978), no1711, 425-461.
 
\bibitem[BDV09]{barberis-dotti-verbitsky} M. L. Barberis, I. G. Dotti, M. Verbitsky, 
\emph{Canonical bundles of complex nilmanifolds, with applications to hypercomplex geometry, }
 Math. Res. Lett. \textbf{16} (2009), no. 2, 331--347.  
 
\bibitem[Baz18]{bazzoni} G.~Bazzoni,
\emph{Locally conformally symplectic and K\"ahler geometry},
EMS Surv. Math. Sci. \textbf{5} (2018), no. 1-2, 129--154. 
   
  \bibitem[BT82]{BT82} E.~Bedford, B.~A. Taylor, 
\emph{A new capacity for plurisubharmonic functions}.
Acta Math. \textbf{149} (1982), no.~1-2, 1--40.  
 
\bibitem[Bel00]{belgun} F. A. Belgun, 
\emph{On the metric structure of non-K\"ahler complex surfaces, }
 Math. Ann. \textbf{317} (2000), no. 1, 1--40.

\bibitem[BG88]{benson-gordon} C. Benson, C. Gordon, 
\emph{K\"ahler and symplectic structures on nilmanifolds, }
 Topology \textbf{27} (1988), no. 4, 513--518.

\bibitem[COUV16]{couv} M. Ceballos, A. Otal, L. Ugarte, R. Villacampa,
\emph{Invariant complex structures on $6$-nilmanifolds: classification, Fr\"olicher spectral sequence and special Hermitian metrics, }
 J. Geom. Anal. \textbf{26} (2016), no. 1, 252--286.

 \bibitem[Chi14]{Chi14} I. Chiose, 
 \emph{Obstructions to the existence of K\"ahler structures on compact complex manifolds.}, 
  Proc. Amer. Math. Soc. 142 (2014), no. 10, 3561-3568.

 \bibitem[Chi16]{Chi13} I. Chiose, 
 \emph{The K\"ahler rank of compact complex manifolds}, 
 J. Geom. Anal. 26 (2016), no. 1, 603-615.
 
 \bibitem[Chi16b]{Chi16} I. Chiose,
\emph{On the invariance of the total Monge-Amp\`ere volume of hermitian metrics},
Preprint arXiv:1609.05945.  

\bibitem[CGZ13]{CGZ13} D.Coman, V.Guedj, A.Zeriahi:
{\it Extension of plurisubharmonic functions with growth control. }
Journal f\"ur die reine und angewandte Mathematik, {\bf 676} (2013), 33-49.
  
 \bibitem[Dem92]{Dem92}  J.P. Demailly,
 {\it Regularization of closed positive currents and interSection theory}.
    J. Algebraic Geom. {\bf 1} (1992), no. 3, 361--409.
    
  \bibitem[Dem]{Dem12}  J.P. Demailly,  
   \emph{Analytic methods in algebraic geometry},
Surveys of Modern Mathematics, 1. International Press; 
Higher Education Press, Beijing, 2012. viii+231 pp.

    \bibitem[DP04]{DP04}  J.P. Demailly, M. P\u{a}un,
   \emph{Numerical characterization of the K\"ahler cone of a compact K\"ahler manifold. }
 Ann. of Math. (2) 159 (2004), no. 3, 1247--1274. 
 
     \bibitem[DLM17]{DLM17}  G.Deschamps, N.Le Du, C.Mourougane,
   \emph{Hessian of the natural hermitian form on twistor spaces. }
   Bull. Soc. Math. France 145 (2017), no1, 1-7.
  
    \bibitem[Din16]{Din16} S.~Dinew, 
   \emph{Pluripotential theory on compact hermitian manifolds},
   Ann. Fac. Sci. Toulouse Math. (6) 25 (2016), no. 1, 91--139.
   
   \bibitem[DK12]{DK12} S.~Dinew, S.~Ko{\l}odziej,
   \emph{Pluripotential estimates on compact hermitian manifolds. }
   Advances in geometric analysis, 69-86, Adv. Lect. Math. (ALM), 21, Int. Press, 2012. 
  
   \bibitem[DO98]{DO98} S.Dragomir, L.Ornea
   \emph{Locally conformal Kähler geometry.}
 Progress in Mathematics, 155. Birkh\"auser Boston, Inc., Boston, MA, 1998. xiv+327 pp.
 
  \bibitem[Eg01]{Eg01} N.Egidi,
\emph{Special metrics on compact complex manifolds. }
 Differential Geom. Appl. 14 (2001), no. 3, 217-234. 

\bibitem[EFV12]{enrietti-fino-vezzoni}
N. Enrietti, A. Fino, L. Vezzoni, \emph{Tamed symplectic forms and strong K\"ahler with torsion metrics}, J. Symplectic Geom. \textbf{10} (2012), no. 2, 203--223.
 
 \bibitem[FG86]{fernandez-gray} M. Fern\'andez, A. Gray, 
 \emph{The Iwasawa manifold, }
 Differential geometry, Pe\~n\'iscola 1985, 157--159, 
 {\em Lecture Notes in Math.}, \textbf{1209}, Springer, Berlin, 1986.

\bibitem[FG04]{fino-grantcharov} A. Fino, G. Grantcharov, 
\emph{Properties of manifolds with skew-symmetric torsion and special holonomy}, Adv. Math. \textbf{189} (2004), no. 2, 439--450.

\bibitem[FGV19]{fino-grantcharov-vezzoni}
A. Fino, G. Grantcharov, L. Vezzoni, \emph{Astheno-K\"ahler and balanced structures on fibrations}, Int. Math. Res. Not. IMRN \textbf{2019} (2019), no. 22, 7093--7117.

\bibitem[FOU15]{FOU15} A. Fino, A.Otal, L.Ugarte,
\emph{Six-dimensional solvmanifolds with holomorphically trivial canonical bundle},
I.M.R.N. 2015, no24, 13757-13799.
 
\bibitem[FPS04]{FPS04} A. Fino, M. Parton, S. Salamon,
\emph{Families of strong KT structures in six dimensions},
Comment. Math. Helv. 79 (2004), no. 2, 317--340. 

\bibitem[FT09]{FT09} A.~Fino and A.~Tomassini, 
\emph{Blow-ups and resolutions of strong {K}\"{a}hler  with torsion metrics}, 
Adv. Math. \textbf{221} (2009), no.~3, 914--935.

\bibitem[FV15]{fino-vezzoni} A. Fino, L. Vezzoni, 
\emph{Special Hermitian metrics on compact solvmanifolds, } 
J. Geom. Phys. \textbf{91} (2015), 40--53.

\bibitem[GL10]{GL10} B. Guan, Q. Li,
\emph{Complex Monge-Amp\`ere equations and totally real submanifolds.}
 Adv. Math. 225 (2010), no. 3, 1185--1223.

\bibitem[GL21a]{GL21a} V.~Guedj, C.~H. Lu, {\it Quasi-plurisubharmonic envelopes 1: Uniform estimates on K\"ahler manifolds},  preprint arXiv:2106.04273. 

\bibitem[GL21b]{GL21b} V.~Guedj, C.~H. Lu,  {\it Quasi-plurisubharmonic envelopes 2: Bounds on Monge-Amp\`ere volumes}, preprint arXiv:2106.04272.  To appear in Algebraic Geometry.

\bibitem[GL21c]{GL21c} V.~Guedj, C.~H. Lu, {\it Quasi-plurisubharmonic envelopes 3: Solving Monge-Amp\`ere equations on hermitian manifolds},  preprint arXiv:2107.01938. 

\bibitem[GZ]{GZbook} V.~Guedj, A.~Zeriahi, 
\emph{Degenerate complex {M}onge-{A}mp\`ere  equations}, 
EMS Tracts in Mathematics, vol.~26, European Mathematical Society  (EMS), Z\"{u}rich, 2017. 

\bibitem[HL83]{HL83}  R.Harvey, B.Lawson,
{\em An intrinsic characterization of K\"ahler manifolds.}
 Invent. Math. 74 (1983), no. 2, 169-198. 

\bibitem[Has89]{hasegawa} K. Hasegawa, 
\emph{Minimal models of nilmanifolds, }
 Proc. A.M.S.. \textbf{106} (1989), no. 1, 65--71.

\bibitem[Hir60]{Hir60}  H.Hironaka, 
{\em On the theory of birational blowing-up.}
hesis, Harvard (1960) (unpublished).

\bibitem[Hit81]{Hit81}  N.J.Hitchin,
{\em K\"ahlerian twistor spaces.}
Proc. London Math. Soc. (3) 43 (1981) 133-150.

\bibitem[Iv04]{Iv04}  S.Ivashkovich,
{\em Extension properties of meromorphic mappings with values in non-K\"ahler complex manifolds.}
 Ann. of Math. (2) 160 (2004), no. 3, 795-837.
  
 \bibitem[KV98]{KV98} D.Kaledin, M.Verbistky,
\emph{Non-hermitian Yang-Mills connections},
Selecta Math. 4 (1998), no2, 279-320.

 \bibitem[Kas13]{Kas13} H.Kasuya
\emph{Vaisman metrics on solvmanifolds and Oeljeklaus-Toma manifolds.}
Bull. Lond. Math. Soc. 45 (2013), no. 1, 15-26. 

\bibitem[KN15]{KN15} S.~Ko{\l}odziej, N.C.~Nguyen,
\emph{Weak solutions to the complex Monge-Amp\`ere equation on compact hermitian manifolds},
Contemp. Math. 644 (2015), 141--158.

   \bibitem[KN19]{KN19} S.~Ko{\l}odziej, N.C.Nguyen,
 \emph{ Stability and regularity of solutions of the Monge-Ampère equation on hermitian manifolds},
Adv. Math. 346 (2019), 264-304. 

    \bibitem[LPT21]{LPT21} C.~H. Lu, T.T. Phung, T.D. T\^o,
 \emph{ Stability and H\"older regularity of solutions to complex Monge-Amp\`ere equations on compact hermitian manifolds.}
  Ann. Inst. Fourier (Grenoble) 71 (2021), no. 5, 2019-2045.

\bibitem[Mag86]{magnin} L. Magnin, 
\emph{Sur les alg\`ebres de Lie nilpotentes de dimension $\leq7$, }
 J. Geom. Phys. \textbf{3} (1986), no. 1, 119--144.

\bibitem[Mal45]{malcev}
A. Malcev, \emph{On solvable Lie algebras}, Bull. Acad. Sci. URSS. Sr. Math. [Izvestia Akad. Nauk SSSR] \textbf{9} (1945), 329--356.
  
\bibitem[Mic82]{Mic82}
M.-L.Michelsohn, \emph{On the existence of special metrics in complex geometry}, Acta Math. 149 (1982), no1, 261-295.

\bibitem[Mil76]{milnor}
J. Milnor, \emph{Curvatures of left invariant metrics on Lie groups}, Adv. Math. \textbf{21} (1976), no. 3, 293--329.

\bibitem[Mor58]{morozov}
V. V. Morozov, \emph{Classification of nilpotent Lie algebras of sixth order}, Izv. Vys\c{s}. U\c{c}ebn. Zaved. Matematika \textbf{1958} (1958) no. 4 (5), 161--171.

\bibitem[Nak75]{nakamura}
I. Nakamura, \emph{Complex parallelisable manifolds and their small deformations,} J. Differential Geometry 10 (1975), 85-112. 

    \bibitem[OV10]{OV10} L.Ornea, M.Verbistky
\emph{Locally conformal K\"ahler manifolds with potential},
Math. Ann. \textbf{348} (2010), 25--33.
 
    \bibitem[OV11]{OV11} L.Ornea, M.Verbistky
  \emph{A report on locally conformally K\"ahler manifolds.}
 Harmonic maps and differential geometry, 135-149,
Contemp. Math., 542, A.M.S., Providence, RI, 2011. 

   \bibitem[OV20]{OV20} L.Ornea, M.Verbistky
  \emph{Hopf surfaces in locally conformally K\"ahler manifolds with potential.}
   Geom. Dedicata \textbf{207} (2020), 219-226.

    \bibitem[OV]{ornea-verbitsky-book} L.Ornea, M.Verbistky, \emph{Principles of Locally Conformally K\"ahler Geometry}.
   
     \bibitem[Ota14]{otal-thesis} A. Otal, 
     {\em Solvmanifolds with holomorphically trivial canonical bundle},
      PhD Thesis, Universidad de Zaragoza, 2014.

   \bibitem[Ot20]{Ot20} A. Otiman,
  \emph{Special hermitian metrics on Oeljeklaus-Toma manifolds},
   	arXiv:2009.02599. To appear in Bulletin of the London Mathematical Society.

  \bibitem[Pop17]{popovici} D. Popovici,
\emph{Volume and self-intersection of differences of two nef classes},
Ann. Sc. Norm. Super. Pisa Cl. Sci. (5) 17 (2017), no. 4, 1255--1299.

\bibitem[Sal82]{Sal82} S. M. Salamon,
\emph{Quaternionic K\"ahler manifolds.}
 Invent. Math. 67 (1982), no. 1, 143-171. 
 
\bibitem[Sal01]{salamon} S. M. Salamon,
\emph{Complex structures on nilpotent Lie algebras, }
 J. Pure Appl. Algebra \textbf{157} (2001), no. 2-3, 311--333.
 
   \bibitem [STW17]{STW17} G. Sz\'ekelyhidi,  V. Tosatti, B. Weinkove, 
 \emph{Gauduchon metrics with prescribed volume form},
 Acta Math. 219 (2017), no. 1, 181--211. 
 
 \bibitem [Taub92]{Taub92} C.H.Taubes
 \emph{The existence of anti-self-dual conformal structures.}
 Journal of Diff. Geometry 36 (1992), no1, 163-253.
 
  \bibitem [TW10]{TW10} V. Tosatti, B. Weinkove,
\emph{The complex Monge-Amp\`ere equation on compact hermitian manifolds.}
 J. Amer. Math. Soc. 23 (2010), no. 4, 1187--1195.
 
  \bibitem[Uga07]{ugarte-transf} L. Ugarte, 
  \emph{Hermitian structures on six-dimensional nilmanifolds, }
  Transform. Groups \textbf{12} (2007), no. 1, 175--202.

\bibitem[UV14]{ugarte-villacampa} L. Ugarte, R. Villacampa, 
\emph{Non-nilpotent complex geometry of nilmanifolds and heterotic supersymmetry, }
 Asian J. Math. \textbf{18} (2014), no. 2, 229--246.
 
   \bibitem [V82]{V82} I.Vaisman
\emph{ Generalized Hopf manifolds.}
 Geom. Dedicata 13 (1982), no. 3, 231-255.
 
   \bibitem [Verb14]{Verb14} M.Verbitsky,
\emph{ Rational curves and special metrics on twistor spaces.}
Geometry and Topology 18 (2014), 897-909.

   \bibitem [Wan54]{Wan54}  H.-C.Wang
\emph{ Closed manifolds with homogeneous complex structure.}
Amer. J. Math. 76 (1954), 1-32. 
 
\bibitem[Yac98]{yachou} A.~Yachou,
\emph{Sur les vari\'et\'es semi-k\"ahl\'eriennes},
PhD Thesis, Universit\'e de Lille, 1998.
 
\bibitem [Yau78]{Yau78}  S.~T.~Yau,
\emph{On the Ricci curvature of a compact K{\"a}hler manifold and the complex Monge-Amp{\`e}re equation. I.}
 Comm. Pure Appl. Math. {\bf 31} (1978), no. 3, 339--411.  

\end{thebibliography}
\end{document}